\documentclass[12pt]{article}
\usepackage{latexsym}
\usepackage{amssymb}
\usepackage{amsmath}
\usepackage{mathrsfs}
\usepackage[pdftex]{graphicx}
\usepackage{rawfonts}
\input{prepictex}
\input{pictex}
\input{postpictex}
\usepackage[OT2,OT1]{fontenc}
\def\cyr{%
\renewcommand\rmdefault{wncyr}%
\renewcommand\sfdefault{wncyss}%
\renewcommand\encodingdefault{OT2}%
\normalfont
\selectfont}
\DeclareTextFontCommand{\textcyr}{\cyr}

\newcommand{\rad}{\text{\cyr   ya}}
\def\Bbb{\mathbb}

\def\CC{\mathbb C}

\def\Riem{{\mathcal R}}

\newtheorem{main}{Theorem}
\DeclareMathOperator{\Hess}{Hess}

\DeclareMathOperator{\Aut}{Aut}

\DeclareMathOperator{\Int}{Int}

\DeclareMathOperator{\area}{area}

\newtheorem{thm}{Theorem}
\newtheorem{lem}[thm]{Lemma}
\newtheorem{prop}[thm]{Proposition}
\newtheorem{cor}[thm]{Corollary}

\newenvironment{proof}{\medskip \noindent
{\bf Proof.}}{\hfill \rule{.5em}{1em}
\\}

\def\ZZ{{\mathbb Z}}
\def\RR{{\mathbb R}}
\def\CP{{\mathbb C \mathbb P}}
\begin{document}

\title{On  Conformally K\"ahler, Einstein Manifolds}

\author{Xiuxiong Chen\thanks{Supported 
in part by  NSF grant DMS-0406346.}, 
Claude LeBrun\thanks{Supported 
in part by  NSF grant DMS-0604735.},    and Brian Weber 
  }

%MSC Primary
%53C25
%53C55
%Secondary
%14J26
%14J80
%53C28

\date{April 27, 2007}
\maketitle

\begin{abstract}
We prove that any compact complex surface with $c_1>0$ admits an
Einstein metric which is conformally related to a K\"ahler metric. 
The key
new ingredient is the existence of such a metric on the 
blow-up  ${\CP}_2\# 2\overline{\CP}_2$
of the complex projective plane at two distinct points. 
\end{abstract}

\section{Introduction}

Recall that a Riemannian manifold $(M,h)$ is said to be {\em Einstein} if its 
Ricci tensor $r$ satisfies 
$$r=\lambda h$$
for some real number $\lambda$. If, on the other hand,  $M$ is equipped with an 
integrable almost-complex structure $J$, so that $(M,J)$ is a complex manifold, 
then we say that a Riemannian metric $h$ is {\em Hermitian} with respect to $J$ 
if $h= h(J\cdot , J \cdot)$.
The purpose of this article is to prove the following: 

\begin{main}\label{big}
If  $(M,J)$ is the compact complex surface 
obtained from $\CP_2$ by blowing up 
two distinct points, then $(M,J)$ admits an Einstein metric $h$ of positive Ricci curvature 
which is 
Hermitian with respect to $J$. 
\end{main}
The complex surface figuring in this result
 is diffeomorphic to $\CP_2\# 2\overline{\CP}_2$,
and can also be obtained by blowing up  $\CP_1\times \CP_1$ at one 
point. This manifold has $c_1(M,J)> 0$, in the sense that  its first Chern class
is  the K\"ahler class of a K\"ahler metric; thus, it is an example of a del 
Pezzo surface---i.e. a 
Fano manifold of complex dimension $2$. However, by a result of Matsushima
\cite{mats}, 
it cannot admit a 
K\"ahler-Einstein metric, because its automorphism group is non-reductive. 

Our strategy for proving Theorem \ref{big} was originally motivated 
by the work of Derdzi{\'n}ski \cite{bes,derd} on 
Einstein metrics which are {\em conformally}
 K\"ahler. By extending Derdzi{\'n}ski's results, 
the second author  has  shown elsewhere  \cite{leb4}  that if a Hermitian 
metric $h$  on a compact complex surface $(M^4,J)$ is  Einstein,
  then $h$ is necessarily conformal
to a K\"ahler metric $g$, and that, unless  $h$ is itself K\"ahler,  then 
\begin{itemize}
\item $(M,J)$ has
$c_1> 0$, and is obtained from $\CP_2$
by blowing up $1$, $2$, or $3$ points in general position;
\item $h$ has positive Ricci curvature; 
\item $g$ is an extremal K\"ahler metric in the sense of Calabi  \cite{calabix,calabix2};
\item the scalar curvature $s$ of $g$ is  everywhere positive; and
\item after appropriate normalization, $h=s^{-2}g$.
\end{itemize} 
These observations conversely 
motivate  the proof of Theorem \ref{big}, which proceeds by constructing an
extremal K\"ahler metric $g$ with the property that $h=s^{-2}g$ is Einstein. 
This is done by using a weak compactness result of the 
first and third authors to produce large deformations of certain
extremal K\"ahler metrics constructed by Arezzo, Pacard and Singer \cite{arpasing}.
For a concise summary of the proof, see \S \ref{stratego} below.

Now $\CP_2\# \overline{\CP}_2$ carries an Einstein metric 
originally discovered by Page  \cite{page}, and it was later pointed out 
\cite{bes} that 
the Page metric is actually   conformal
 to one of the extremal K\"ahler metrics 
constructed  by Calabi \cite{calabix}
on 
 the one-point blow-up $\CP_2$. 
As the 
K\"ahler-Einstein case has been completely solved by 
Aubin, Yau, and Tian \cite{aubin,yau,tian}, Theorem \ref{big}
 exactly fills in the missing puzzle-piece needed to prove  the following:

\begin{cor} \label{free} 
Let $(M^4,J)$ 
be a compact complex surface. Then $M$ admits
an   Einstein metric which is Hermitian with respect to $J$ if and only if 
$$c_1 (M,J)= \lambda [\omega ]$$ for
some real constant $\lambda$ and some K\"ahler class $[\omega] \in H^2 (M, \RR)$.
\end{cor}

Theorem \ref{big} also completes the solution of a slightly different problem. 
Let us instead focus on the underlying $4$-manifold $M$ of  a compact
complex surface, and, without supposing anything about the 
relationship between the metric and complex structure, ask when this 
smooth manifold 
admits an Einstein metric with positive $\lambda$. By 
 the Hitchin-Thorpe inequality \cite{bes,hit,tho}, the existence of such a metric  implies that 
$M$  has $c_1^2 = 2\chi + 3\tau > 0$. 
However, the latter ensures \cite{FM,spccs} 
that the Seiberg-Witten invariant \cite{witten} is well-defined, and the existence of 
a positive-scalar-curvature metric then forces the invariant to vanish. 
But since $c_1^2 > 0$, the Kodaira classification  \cite{bpv} says that the complex
surface $M$
is either rational or of general type. Since the Seiberg-Witten invariant of $M$ would
be non-zero if it were of general type, we therefore conclude  
that $M$ can  obtained from either 
$\CP_2$ or a rational ruled surface by blowing up; and since $c_1^2 > 0$, we thus 
conclude that 
$M$ is diffeomorphic to either ${\CP}_2\# k\overline{\CP}_2$,  $0\leq k \leq 8$, or 
to $S^2 \times S^2$. 
Similarly, one can reach this same conclusion if  the assumption 
that $M$ admits a complex structure is replaced with the hypothesis that it admits a 
symplectic form  \cite{liu1,ohno}. 
In conjunction with the 
results  of Tian-Yau \cite{ty}, Theorem \ref{big} therefore implies the following:

\begin{cor}
Let $M$ be a smooth compact oriented $4$-manifold which either 
 admits 
 a complex structure or    admits a symplectic structure. 
Then $M$ carries an Einstein metric of positive
scalar curvature if and only if  it 
is orientedly diffeomorphic   
 to  either  a connected  sum 
${\CP}_2\# k\overline{\CP}_2$, where  $0\leq k \leq 8$, or else 
to $S^2 \times S^2$. 
\end{cor}

\section{Strategy}
 \label{stratego} 

We now  outline the  proof of Theorem \ref{big}. 

Arezzo, Pacard, and Singer \cite{arpasing} have shown that 
if $\CP_1 \times \CP_1$ is blown up at a point, the resulting
complex surface $M$
\begin{center}
\begin{picture}(240,80)(0,3)
\put(-7,70){\line(1,0){54}}
\put(40,75){\line(2,-3){28}}
\put(0,75){\line(-2,-3){28}}
%\put(-7,5){\line(1,0){54}}
\put(68,44){\line(-1,-1){52}}
\put(-28,44){\line(1,-1){52}}
\put(20,77){\makebox(0,0){$E$}}
\put(-21,13){\makebox(0,0){$F_1$}}
\put(60,13){\makebox(0,0){$F_2$}}
\put(100,40){\vector(1,0){50}}
\put(210,0){\line(1,1){45}}
\put(220,0){\line(-1,1){45}}
\put(210,80){\line(1,-1){45}}
\put(220,80){\line(-1,-1){45}}
\put(215,75){\circle*{4}}
\end{picture}
\end{center}
 admits extremal K\"ahler metrics; in particular, 
their work  shows that such metrics can be found
in the K\"ahler classes $F_1+F_2 -\epsilon E$ for 
any sufficiently small $\epsilon > 0$, where $E$ is the
Poincar\'e dual of the  exceptional divisor introduced
by blowing up, and where $F_1$ and $F_2$ are the Poincar\'e duals 
of the factor $\CP_1$'s of  $\CP_1 \times \CP_1$. 
We note in passing that  the homology classes
$F_1-E$ and $F_2-E$ are also represented by $(-1)$-curves,
and that blowing these two exceptional divisors down 
\begin{center}
\begin{picture}(240,80)(0,3)
\put(-7,70){\line(1,0){54}}
\put(40,75){\line(2,-3){28}}
\put(0,75){\line(-2,-3){28}}
\put(68,44){\line(-1,-1){52}}
\put(-28,44){\line(1,-1){52}}
\put(20,77){\makebox(0,0){$E$}}
\put(-21,13){\makebox(0,0){$F_1$}}
\put(60,13){\makebox(0,0){$F_2$}}
\put(100,40){\vector(1,0){50}}
\put(210,0){\line(2,3){52}}
\put(220,0){\line(-2,3){52}}
\put(165,70){\line(1,0){100}}
\put(172.6,70.5){\circle*{4}}
\put(257.4,70.5){\circle*{4}}
\end{picture}
\end{center}
results
in  $\CP_2$; thus $M$ may also  be described 
as $\CP_2\# 2\overline{\CP}_2$. We also note that the K\"ahler
classes we are choosing to study are ones for which 
$F_1$ and $F_2$ have equal areas, even though the Arezzo-Pacard-Singer 
result would also construct extremal K\"ahler metrics for which
the ratios of these areas is quite  arbitrary. By the uniqueness
of extremal K\"ahler metrics in a given K\"ahler class \cite{xxgang},
the metrics we are considering therefore not only have an isometric 
$U(1)\times U(1)$-action, but also  admit an additional isometric
$\ZZ_2$ action which interchanges $F_1$ and $F_2$.  
This leads to major technical simplifications
which will play a crucial r\^ole in our proof. 
We thus introduce the  term {\em bilaterally symmetric} to 
describe both those  K\"ahler classes which are invariant under the
interchange $F_1\leftrightarrow F_2$, as well as the  extremal K\"ahler metrics
we will find in many such classes. 

By a  general result  \cite{ls2} proved via the 
inverse-function theorem, the  {\em extremal cone},
consisting of the K\"ahler classes of all extremal K\"ahler metrics
on $(M,J)$ 
is automatically  open in $H^{1,1}(M, \RR )$; consequently,  
the set of $\epsilon$ for
which the relevant K\"ahler class contains an extremal K\"ahler metric is open.
As we increase $\epsilon$, we can then use the Futaki invariant to  show that the 
value of the Calabi functional 
$$
{\mathcal C}(g) = \int_Ms^2_gd\mu_g
$$
on these extremal metrics 
initially decreases, but would eventually reach a minimum 
and then increase if we could simply take $\epsilon$ to be sufficiently large.  
If we can simply arrange for $\epsilon$ to achieve a value which extremizes ${\mathcal C}(g)$,
we  show in \S \ref{fugue} that the corresponding extremal metric
will then  actually be conformally Einstein.

Thus, the problem essentially boils down to showing that, within a certain range, 
the set of $\epsilon$ achieved by extremal K\"ahler metrics is 
actually closed as well as open. Our method of showing this is based on an 
orbifold compactness result proved elsewhere by the first and third authors \cite{chenweb}.
In order to apply this, we must first prove a uniform estimate for 
the Sobolev constant of the metrics involved; this is done in \S \ref{sobstory}. 
Next, we must show that orbifold singularities cannot form in the limit.
This is done by showing that curvature can never concentrate in too small
a region, since, upon rescaling, this would result in an asymptotically
locally flat manifold which, given the topological and symmetry
conditions imposed by our situation,  would ultimately require the
concentration of more curvature than is actually available.  

\section{The Calabi Functional}
 \label{action} 

If $(M,g,J)$ is an extremal K\"ahler metric on a compact complex surface, 
the Calabi functional takes the value
$${\mathcal C}(g) = s_0^2\int  d\mu + \int (s-s_0)^2 d\mu=
32\pi^2\frac{(c_1\cdot [\omega ])^2}{[\omega ]^2}
-{\cal F}(\xi, [\omega ])
$$
where $s_0$ is the average value of the 
scalar curvature,
$\cal F$ denotes the Futaki invariant, and
$\xi = \mbox{grad}^{1,0} s$ is the extremal
vector field of the class $[\omega ]$. 
It is crucial for our purposes that $\xi$ 
may be determined \cite{fuma1}  up to conjugation
even  without knowing that an  extremal metric 
exists. Thus, one may define a functional 
$$
{\mathcal A} ([\omega ]) = \frac{(c_1\cdot [\omega ])^2}{[\omega ]^2}- \frac{1}{32\pi^2} 
{\mathcal F}(\xi, [\omega ])
$$
on the entire K\"ahler cone, independent of the existence of 
extremal K\"ahler metrics. This functional has the important
property \cite{xxel}
that any K\"ahler metric $g$ in the K\"ahler class $[\omega]$ satisfies
the curvature inequality 
$$\frac{1}{32\pi^2}\int s^2 d\mu \geq {\mathcal A}([\omega ])$$
with equality iff $g$ is an extremal metric. Notice that our
normalization has been chosen  so that we automatically have
$${\mathcal A}([\omega ])\geq c_1^2 (M)$$
for any K\"ahler class. This section will now begin with a discussion of  
the problem, first explored in \cite{spccs}, of  finding a critical point of  ${\mathcal A}$,
considered  as 
a function on the K\"ahler cone. To do  this, we will use computations of the 
Futaki invariant first given in  \cite{ls} for the blow-up
of ${\Bbb CP}_2$ at $\leq 3$ points  in general position.

\def\a{\alpha}
\def\b{\beta}
\def\d{\varepsilon}
Any extremal K\"ahler metric is invariant \cite{calabix2} under a maximal compact subgroup
of the identity component of the complex automorphism group, and 
since such subgroups are unique up to 
conjugation, we may simply choose one; in the present case, this means
that we may consider only metrics which are invariant under the $2$-torus
$T^2$ of automorphisms of $M$ induced by 
$$
([u_1:u_2], [v_1:v_2]) \mapsto ([u_1:e^{i\theta} u_2], [v_1:e^{i\phi} v_2]) ~,
$$
where $M$ is thought of as the blow-up of $\CP_1 \times \CP_1$ at $([0:1], [0:1])$. 
We also choose  only  to consider  {\em bilaterally symmetric} K\"ahler classes 
$[\omega ] =  (\beta +\varepsilon) (F_1 + F_2) - \varepsilon E$ on 
 $M= \CP_2 \# 2\overline{\CP}_2$: 
\begin{center}
\begin{picture}(100,90)(0,-10)
\put(-7,70){\line(1,0){54}}
\put(40,75){\line(2,-3){28}}
\put(0,75){\line(-2,-3){28}}
\put(68,44){\line(-1,-1){52}}
\put(-28,44){\line(1,-1){52}}
\put(20,77){\makebox(0,0){$\d$}}
\put(-23,56){\makebox(0,0){$\beta$}}
\put(60,56){\makebox(0,0){$\beta$}}
\put(-21,13){\makebox(0,0){$\beta+\d$}}
\put(60,13){\makebox(0,0){$\beta+\d$}}
\end{picture}
\end{center}
Here  the term {\em bilaterally symmetric} is again used to indicate that the 
class in question is invariant under $F_1\leftrightarrow F_2$. 
The numbers  $\beta$ and $\varepsilon$ 
respectively represent the areas of the $(-1)$-curves 
$F_1-E$ and $E$; 
 both are thus required to be positive, 
but they may otherwise be taken to be completely arbitrary. 
Since $\cal A$ is invariant under the 
${\Bbb Z}_2$-action $F_1\leftrightarrow F_2$ induced by 
interchanging the factors of $\CP_1\times \CP_1$, 
and is also invariant under rescaling $[\omega ] \rightsquigarrow a [\omega]$, 
any critical point of the function 
$$f(x) = {\mathcal A} ([1+x] (F_1+F_2) - xE)$$
will yield a critical point $(\beta , \varepsilon ) = (1, x)$,
and conversely, up to rescaling,  a critical point of ${\mathcal A}$ arises this way if
and only if the relevant K\"ahler class is bilaterally symmetric.  

Now, for any $T^2$-invariant,  
bilaterally symmetric K\"ahler metric, the real part of the extremal K\"ahler  vector field $\xi$
belongs to the Lie algebra of our maximal compact subgroup $T^2\subset \Aut_0(M)$, and 
must be invariant under $F_1\leftrightarrow F_2$. 
Thus $\xi$  must be a multiple of the generator 
 $\Xi$ 
of the ${\Bbb C}^\times$-action induced by the action 
$$
([u_1:u_2], [v_1:v_2]) \mapsto ([u_1:\zeta u_2], [v_1:\zeta v_2])
$$
on $\CP_1\times \CP_1$. 
But $\Xi = \mbox{grad}^{1,0}t$ for a real-valued 
Hamiltonian function $t$ which by symplectic reduction \cite{ls}
can be shown to satisfy 
\begin{eqnarray} [12\pi \omega ]^2 \int (t-t_0)^2 d\mu &= &
 12\b^6  +72\b^5\d 
+ 138\b^4\d^2  + 120\b^3\d^3
\nonumber
\\ &&\label{rhs} 
+ 54\b^2\d^4 + 12\b\d^5 +\d^6
\end{eqnarray}
where $t_0$ is the average value of $t$. 
On the other hand, it was shown in \cite{ls} that 
$$[\omega ]^2{\cal F}(\Xi, [\omega ])= 4\b\d 
\left [\frac{\d^2}{3}+
\b\d+
\b^2 \right] . $$
Since ${\cal F}(\Xi, [\omega ])=-\int (t-t_0) (s-s_0)d\mu$,
an explicit formula for ${\cal A}$ can now be deduced by 
setting $(s-s_0)= \lambda (t-t_0)$ and solving for
$\lambda$ to obtain
$$\lambda = -  \frac{(12\pi)^24\b\d [\d^2/3+ \b\d+\b^2]}{12\b^6  +72\b^5\d 
+ 138\b^4\d^2  + 120\b^3\d^3
+ 54\b^2\d^4 + 12\b\d^5 +\d^6}$$
so that 
\begin{eqnarray*}
-\frac{1}{32\pi^2}{\cal F}(\xi, [\omega ])&=& -\frac{1}{32\pi^2}\lambda {\cal F}(\Xi, [\omega ])\\ &=&
 \frac{9}{2[\omega]^2}
 \frac{ \Big(4\b\d [\d^2/3+ \b\d+\b^2]\Big)^2}{12\b^6  +72\b^5\d 
+ 138\b^4\d^2  + 120\b^3\d^3
+ 54\b^2\d^4 + 12\b\d^5 +\d^6}
\end{eqnarray*}
and
\begin{eqnarray*}
{\mathcal A}([\omega ])
 &=&\frac{(c_1\cdot [\omega ])^2}{[\omega]^2} -\frac{1}{32\pi^2}{\cal F}(\xi, [\omega ])
\\ &=& \frac{(4\b+3\d)^2}{[\omega]^2} -\frac{1}{32\pi^2}{\cal F}(\xi, [\omega ])
\end{eqnarray*}
Hence 
$${\mathcal A}([\omega ]) = f(x) = 3\left(
\frac{32+ 176x+318x^2+280x^3+132x^4+32x^5+3x^6}{
12+72x+138x^2+120x^3+54x^4+12x^5+x^6}\right)
$$
where $x=\d/\b$. For $x> 0$, technology\footnote{Figure generated by
Pacific Tech's Graphing Calculator program for Mac OS X.}  indicates that this has a
 unique critical point,
an absolute minimum, at $x\approx 0.958.$ 

\bigskip

\bigskip

\includegraphics{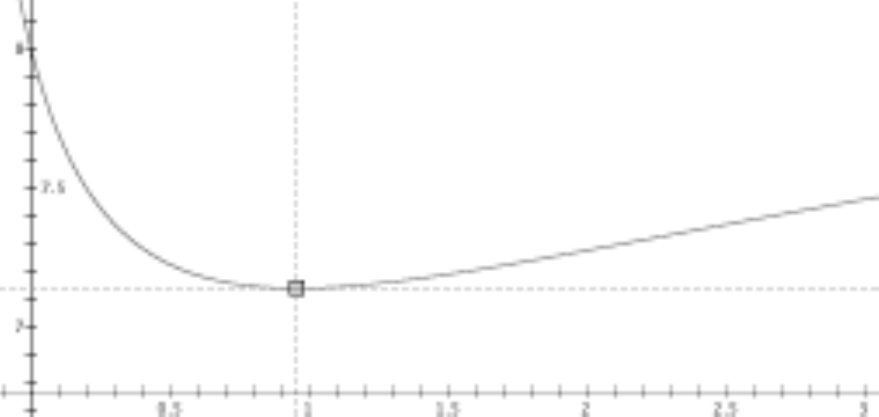}

\bigskip

\bigskip 

\noindent 
However, we will need a great deal less for the purposes of our proof:

\begin{lem} \label{magna} 
There is a number $x_0 >0$ such that the function 
$$f(x) = 
 3\left(
\frac{32+ 176x+318x^2+280x^3+132x^4+32x^5+3x^6}{
12+72x+138x^2+120x^3+54x^4+12x^5+x^6}\right)
$$
has a critical point at $x=x_0$, and such that 
$f(x) < 8$ on $(0,x_0]$. 
\end{lem}
\begin{proof}
Notice that $f(0)= 8$, and that $f^\prime (0)= -4 < 0$, 
so $f(x) < 8$ for small positive $x$. However, 
$\lim_{x\to \infty} f(x) = 9$, so $f^\prime(x)$ must be positive somewhere. 
We can therefore simply define
 $x_0$ be the first positive number at which $f^\prime (x)=0$, since 
 it then follows that $f$ is decreasing on $[0,x_0]$.
  \end{proof}

We leave it as an exercise for the interested reader to check 
that $f^\prime(1)
%= 3\cdot 1152/(409)^2 
> 0$,
so that in fact $x_0 < 1$. 
However, it turns out that we will never actually need this sort of information
for the purposes of our proofs.

 We would also like  to know if there are  values of $x= \d/\b$ for
 which the scalar curvature
 of the corresponding K\"ahler metric is everywhere  positive. To determine this, notice that,
 since the Hamiltonian
 $t$ generates rotations of period $2\pi$ of 
 $F_1$ and $F_2-E$,  while leaving $E$ fixed, 
symplectic geometry tells us that 
 $$(\b + \d ) + \b = 2\pi (t_{\max}-t_{\min}) , $$ 
 and hence  that 
 $$
 s_{\max}-s_{\min}= \frac{ |\lambda| }{2\pi} (2\b + \d)~. 
 $$
Thus 
 \begin{eqnarray*}
s_{\min} &=&  s_{\max} - \frac{ |\lambda| }{2\pi} (2\b + \d)\\
&>& s_0- \frac{ |\lambda| }{2\pi} (2\b + \d)\\
&=&4\pi \frac{c_1\cdot {[\omega]}}{{[\omega]^2/2}}- \frac{ |\lambda| }{2\pi} (2\b + \d)\\
&=& \frac{8\pi (4\b+3\d)}{{2\b^2+4\b\d+\d^2}}
-  \frac{72\pi (2\b + \d)4\b\d [\d^2/3+ \b\d+\b^2]}{12\b^6  +72\b^5\d 
+ 138\b^4\d^2  + 120\b^3\d^3
+ 54\b^2\d^4 + 12\b\d^5 +\d^6}\\
&=& {8\pi}\b^{-1} \Big[ 
\frac{4+3x}{{2+4x+x^2}}
-  \frac{9(2 + x)4x [x^2/3+ x+1]}{12  +72x
+ 138x^2  + 120x^3 + 54x^4 + 12x^5 +x^6}\Big]\\
&=& {8\pi}\b^{-1}\frac{48+180x+264x^2+ 270x^3+  204x^4+ 102x^5+28x^6+3x^7}{(2+4x+x^2)(12  +72x
+ 138x^2  + 120x^3 + 54x^4 + 12x^5 +x^6)}\\
&>&0.
\end{eqnarray*}
This proves the following: 

\begin{lem} \label{rainmaker}
Any bilaterally symmetric extremal K\"ahler metric on 
$M=\CP_2\# 2\overline{\CP}_2$ has strictly positive scalar curvature. 
\end{lem} 

This computation also implies a $C^0$ estimate for the 
scalar curvature of such metrics. Indeed, we now have 
$$
s_{\max} =  s_{\min} + \frac{ |\lambda| }{2\pi} (2\b + \d)< 
s_0 + \frac{ |\lambda| }{2\pi} (2\b + \d) < 2s_0.$$
Letting $V= [\omega]^2/2$ denote the total volume, we thus have 
$$
s_{\max}V^{1/2} < 2s_0V^{1/2} =  \frac{8\pi c_1\cdot [\omega]}{\sqrt{[\omega]^2/2}}=
\frac{8\pi (4\beta + 3\varepsilon)\sqrt{2}}{\sqrt{2\beta^2 + 4\beta \varepsilon + \varepsilon^2}} < 
24\pi\sqrt{2} ,$$
so that we have the following: 

\begin{lem} \label{zest} The scalar curvature of any 
bilaterally symmetric extremal K\"ahler metric $g$ on $M=\CP_2 \# 2\overline{\CP}_2$
satisfies the $C^0$ estimate 
$$|s| V^{1/2} < 24\pi\sqrt{2} .$$
\end{lem}

\section{The Bach Tensor}
 \label{fugue} 

If $M$ is any smooth compact oriented $4$-manifold, consider the 
conformally invariant Riemannian functional 
$${\mathcal W}(g)= \int_{M} |W|^{2}_gd\mu_{g}$$
obtained  by squaring the $L^2$-norm of the Weyl curvature. 
For any  smooth 1-parameter family of metrics 
$$g_t:= g +t\dot{g}+ O(t^2)$$ 
the first variation of this functional is then   given by 
	$$\left. \frac{d}{dt}{\mathcal W}(g_{t})\right|_{t=0} =
	\int {\dot{g}}^{ab} B_{ab} ~d\mu $$
where \cite{bes} the {\em Bach tensor} $B$ is given by
$$B_{ab}:=(\nabla^c\nabla^d+\frac{1}{2}r^{cd})
W_{acbd}~.$$  

The Bach tensor is automatically symmetric and trace-free, and the latter  is
precisely  the infinitesimal
version of the fact that the functional $\int|W|^2d\mu$ is conformally invariant. 
Similarly, since  $\int|W|^2d\mu$ is also invariant under the action of
the diffeomorphism group, the first variation of ${\mathcal W}$ 
with respect to any 
 Lie derivative  $\dot{g}={\mathcal L}_vg$ must also  vanish. Thus 
$$0 = \int (\nabla^{(a}v^{b)})  B_{ab} ~d\mu = -\int v^b (\nabla^a B_{ab}) ~d\mu$$
for any vector field $v$, and it follows that we must have 
$$\nabla^aB_{ab}=0.$$
Thus, the Bach tensor of any metric  is automatically divergence-free. 

Now the Bianchi identities imply that 
$$(\nabla^a\nabla^b + \frac{1}{2}r^{ab}){(\star W)}_{cabd}  =0  $$
for any Riemannian 4-manifold, so we can also rewrite the Bach tensor as 
$$B_{ab}= (2\nabla^c\nabla^d+r^{cd})(W_+)_{acbd}$$
where $W_+= (W+\star W)/2$ is the self-dual Weyl curvature.
Moreover, since 
$$2\int_M | W_+|^2 d\mu = \int_{M} |W|^{2}d\mu + 12\pi^2 \tau (M) $$
for any Riemannian metric, it follows that the vanishing of the Bach tensor
$B$ is equivalent to $g$ being a critical point of the functional 
$${\mathcal W}_+(g)= \int |W_+|_g^2 d\mu_g~.$$

But for a K\"ahler surface $(M^4,g,J)$ with K\"ahler form $\omega$, 
$${(W_+)_{ab}}^{cd}= \frac{s}{12}\left[
\omega_{ab}\omega^{cd} - \delta_{a}^{[c} \delta_{b}^{d]}+ {J_a}^{[c}{J_b}^{d]}
\right]
$$
so in this case we obtain 
$$B_{ab}= \frac{s}{12}\mathring{r}_{ab} + \frac{1}{4}{J_a}^c{J_b}^d\nabla_c\nabla_d s
-\frac{1}{12} \nabla_a\nabla_b s + \frac{1}{6}g_{ab}\Delta s.$$
If $g$ happens to be extremal, $\nabla\nabla s$ is $J$-invariant, and
this simplifies to become
\begin{equation}\label{jose} 
B = \frac{1}{12} \Big[ s \mathring{r} + 2 \Hess_0(s)\Big] 
\end{equation}
where $\Hess_0$ denotes the trace-free part of the Hessian $\nabla\nabla$. 
Moreover, the $J$-invariance of $\nabla\nabla s$ implies that 
$$(\nabla \nabla s ) (J\cdot , \cdot ) = i\partial \bar{\partial}s,$$
so we deduce the following: 
\begin{lem} \label{canon} 
For any extremal K\"ahler  surface 
$(M^4,g,J)$, 
 the Bach tensor $B$ of $g$ 
 can be written as $B= \psi (\cdot , J\cdot )$, where 
 $\psi$ is a 
harmonic  $(1,1)$-form. Moreover, $\psi$ is given explicitly by   
$$
\psi := B(J\cdot , \cdot ) = \frac{1}{12}\Big[ s\rho + 2  i\partial \bar{\partial}s
\Big]_0
$$
where $\rho$ is the Ricci form of $(M,g,J)$, and  $[~\cdot ~ ]_0$ denotes
projection to the primitive    part of a $(1,1)$-form. 
\end{lem} 
\begin{proof} It only remains to show that $\psi$ is harmonic. 
Since $B$ is divergence-free, 
$$(\delta\psi)_b= -\nabla^a\psi_{ab} = \nabla^a\psi_{ba} =  \nabla^a B_{ca}{J_b}^c= 
  {J_b}^c\nabla^a B_{ac}=0~,$$
and  the $(1,1)$-form $\psi$ is  therefore co-closed. But since $B$ is orthogonal to $g$,
$\psi\in \Lambda^{1,1}$ is orthogonal to the K\"ahler form 
$\omega$, and so belongs to the primitive $(1,1)$-forms
$\Lambda^{1,1}_0\subset  \Lambda^{1,1}$. However, 
$\Lambda^{1,1}_0= \Lambda^-$ on a Hermitian manifold of real dimension $4$,
so this shows that $\psi$ is an anti-self-dual $2$-form.Thus 
$$d\psi = d (-\star \psi) = \star \delta \psi =0,$$
and  the $(1,1)$-form  $\psi$ is therefore harmonic, as claimed. 
\end{proof}

Now consider the Calabi functional 
$${\mathcal C} (g) =  \int s^2 d\mu$$
on the space of  K\"ahler metrics. If we restrict this functional to any 
particular  K\"ahler class, the critical points are by definition 
just the extremal K\"ahler metrics. 
However, if we instead have a critical point of this functional on the space
of {\em all} K\"ahler metrics, a critical point must not only be extremal, but
must be conformally Einstein on the set where $s\neq 0$. 

\begin{prop} \label{gloria} 
Suppose that $g$ is an extremal K\"ahler metric on a compact
complex surface, and suppose that its K\"ahler class
 is a critical point of ${\mathcal A}([\omega ])$, considered  
 as a function on the K\"ahler cone. 
Then $g$  has vanishing
Bach tensor. Moreover, on the open subset of $M$ where the scalar curvature
$s$ of $g$ is non-zero, the conformally related metric 
$h=s^{-2}g$ is Einstein. 
\end{prop}
\begin{proof}
Our hypothesis is equivalent to the assumption that $g$
is a critical point of ${\mathcal C}$, considered as
a function on the space of all K\"ahler metrics. 
But for any K\"ahler metric in real dimension $4$, 
$$|W_+|^2 = \frac{s^2}{24},$$
so this happens iff $g$ is a critical point of 
the restriction of ${\mathcal W}_+$ to the space of K\"ahler metrics. 
In other words, our hypothesis is true  if and only if 
$g$ is a K\"ahler metric such that 
$$\int {\dot{g}}^{ab} B_{ab} ~d\mu =0$$
for every $\dot{g}$ arising from a variation through K\"ahler metrics. 
But Lemma \ref{canon} says that one may find such a 
 variation, with  $\dot{\omega} = \psi$, 
by setting $\dot{g}=B$. It therefore follows that 
$$0=\int {\dot{g}}^{ab} B_{ab} ~d\mu =
	\int |B|^2 ~d\mu $$
	and we must therefore  have $B\equiv 0$.

	Now  recall that if $\hat{g}= u^2 g$ is any 
	conformal rescaling of a  given Riemannian metric  $g$, 
	the trace-free Ricci curvature of $\hat{g}$ is
	given by 
	$$\hat{\mathring{r}}= \mathring{r} + (n-2) u \Hess_0 (u^{-1})~,$$
	where $n$ is the real dimension,
	so that $n=4$ in the case at hand. 
	But we have just shown that $B=0$, so 
	 (\ref{jose}) tells us that 
	$$
	\mathring{r} = - 2 s^{-1}\Hess_0(s)~.
	$$
	Setting $u= s^{-1}$, we therefore  conclude that 
	$$
	\hat{\mathring{r}}= - 2 s^{-1}\Hess_0(s) + (4-2) s^{-1}\Hess_0(s)=0~, 
	$$
	so   the conformally related metric $h=s^{-2}g$ is indeed Einstein 
	on  the  open set  $\{ p\in M~|~s(p)\neq 0\}$ where it is defined. 
\end{proof} 

The reader should note that 
Proposition  \ref{gloria} has  previously been pointed out   by   
 Simanca 
\cite{simstrong}, but, because of the central r\^ole it plays in the present work, 
we have thought it important  to include a self-contained and transparent proof
in this article.  
The  fact that 
 Bach-flat   K\"ahler metrics can be rescaled by their scalar curvatures to yield 
 Einstein metrics has of course been known for much longer, and   is due to 
 Derdzi{\'n}ski \cite{derd}.

\begin{cor} \label{suite} 
Let $x_0$ be the positive real number of  Lemma \ref{magna}. If  
the K\"ahler class $(1+ x_0)(F_1+F_2) - x_0 E$
on $M= \CP_2\# 2\overline{\CP}_2$ contains an extremal K\"ahler metric
$g$, then $h=s^{-2}g$ is an Einstein metric on $M$. 
\end{cor} 

\begin{proof}
The functional ${\mathcal A}([\omega ])$ is invariant under rescalings 
$[\omega ] \to \lambda [\omega ]$ and bilateral symmetries 
$F_1 \leftrightarrow F_2$. Thus any critical point of the restriction of ${\mathcal A}$ 
to the classes of the form $(1+x) (F_1+F_2) - x E$ is also a critical point
of ${\mathcal A}$  itself. But since
$$
 f(x)= {\mathcal A} \left((1+x) (F_1+F_2) - x E\right) ~,
$$
Lemma \ref{magna} exactly tells us  that 
$(1+ x_0)(F_1+F_2) - x_0 E$ is such a critical point.  
Since Lemma \ref{rainmaker} also tells us that such an extremal K\"ahler metric would 
automatically have $s> 0$, Proposition \ref{gloria} then guarantees that 
$h=s^{-2}g$ would be an Einstein metric,  defined on all of  $M$. 
\end{proof} 

\section{Sobolev Constants} 
\label{sobstory}

If $(M, J )$ is a compact complex surface, we will say that a
K\"ahler class $[\omega]$ on $M$ belongs to  the {\em controlled  cone}  if

$$c_1^2 (M) - 
\frac{2}{3} \left[ \frac{(c_1\cdot [\omega ]) ^2}{[\omega]^2} - {\mathcal F} ([\omega ], \xi) \right] 
> 0$$
 This is equivalent to requiring that
 $${\mathcal A} ([\omega ]) < \frac{3}{2}c_1^2~,$$
 so that an extremal K\"ahler metric $g$ in $[\omega]$
 would then satisfy 
$$
\frac{1}{32\pi^2} \int_M s^2 d\mu < \frac{3}{2} c_1^2 ~.
$$

Now  our computations in \S \ref{action} 
for  bilaterally symmetric classes shows that 
\begin{eqnarray*}
{\mathcal A} ([\omega ])&=& 3\left(
\frac{32+ 176x+318x^2+280x^3+132x^4+32x^5+3x^6}{
12+72x+138x^2+120x^3+54x^4+12x^5+x^6}\right)\\
&=& 9\left(
\frac{32+ 176x+318x^2+280x^3+132x^4+32x^5+3x^6}{
36+216x+414x^2+360x^3+162x^4+36x^5+3x^6}\right) \\
&<& 9,
\end{eqnarray*}
as follows by term-by-term comparison of  the numerator and denominator. 
Since   $c_1^2 (\CP_2 \# 2\overline{\CP}_2) = 7$, 
we thus have 
$${\mathcal A} ([\omega ]) < 9 < 10.5  = \frac{3}{2} c_1^2(M)$$
for any bilaterally symmetric class. This shows the following:

\begin{lem}
Any bilaterally symmetric K\"ahler class $[\omega]$ on $M=\CP_2 \# 2\overline{\CP}_2$
belongs to the controlled  cone. 
\end{lem}

By refining  an idea   first  suggested
by 
Gang Tian \cite{tianote}, 
we will  now show  that 
 this   allows us to prove to uniform estimates of the Yamabe constants
 and Sobolev constants  of these metrics; 
  cf. \cite{chenweb,tv1}. Let us first recall that the {\em Yamabe
constant} of 
 a conformal class $[g]$ of Riemannian metrics on a compact
$4$-manifold $M$ is the number 
$$Y_{[g]} 
=\inf_{\hat{g}\in [g]}
\frac{\int_M {s}_{\hat{g}}\,d\mu_{\hat{g}}}{
\sqrt{\int_M d\mu_{\hat{g}}}}.$$
By the celebrated work of Trudinger, Aubin, and Schoen \cite{aubis,lp}
 the  infimum for any conformal class $[g]$
is actually achieved by some  metric, and this 
so-called Yamabe minimizer  $g_Y\in [g]$  necessarily has constant
scalar curvature.

Now the scalar curvature of a metric $\hat{g}= u^2g$ conformal to $g$
satisfies 
$$
s_{\hat{g}}u^3= (6\Delta + s) u~,$$
where $\Delta$ is the positive Laplacian, 
so the Yamabe constant may be re-expressed as 
$$Y_{[g]} =\inf_{u\not\equiv 0}\frac{\int (6|\nabla u|^2 + {s}_gu^2) 
\,d\mu_g}{\left(\int u^4d\mu_g\right)^{1/2}},$$
and notice that we are now allowing ourselves to consider even  those smooth $u$ 
which change sign, since replacing 
$u$ with a positive smoothing of $|u|$ at worst decreases the
quotient on the right.  
If $Y_{[g]}> 0$, we thus have 
$$
\|u\|_{L^4}^2 \leq  \frac{6}{Y_{[g]}} ~\|\nabla u\|^2_{L^2} + \frac{\max{s_g}}{Y_{[g]}} ~\|u\|_{L^2}^2 $$
for  all
$u\in L^2_1$. In particular, if we define \cite{andsd,tv1}  the 
 {\em Sobolev constant} $C_S$
of $g$ to be the smallest constant such that the estimate 
$$\|u\|_{L^4}^2 \leq C_S \left( \|\nabla u\|^2_{L^2} + V^{-1/2} \|u\|_{L^2}^2 \right)$$
holds, where $V$ is the total volume of $(M,g)$, then 
we automatically have 
\begin{equation}
\label{cushion} 
C_S\leq \frac{\max (6 , s_{\max} V^{1/2} )}{Y_{[g]}} 
\end{equation}
for any compact Riemannian $4$-manifold  $(M,g)$ with  $Y_{[g]}>0$.

Now the Gauss-Bonnet and  signature theorems for a 
smooth compact oriented $4$-manifold $M$ imply that 
$$
  \frac{1}{4\pi^2}\int_M\left( \frac{s^2}{24}+2|W_+|^2-
\frac{|\mathring{r}|^2}{2}\right)d\mu_g = (2\chi + 3\tau )(M)
$$
for every Riemannian metric on $M$. If $M$ admits an orientation-compatible
almost-complex
structure, $(2\chi + 3\tau )(M)= c_1^2$, and we therefore have 
$$
\int_M\left( \frac{s^2}{24}+2|W_+|^2 \right)d\mu_g \geq 4\pi^2 c_1^2 . 
$$
However, ${\mathcal W}_+= \int |W_+|^2 d\mu$ is conformally
invariant, so  applying this inequality  to a Yamabe minimizer $g_Y\in [g]$
gives us 
$$
Y_{[g]}^2  \geq 96\pi^2  c_1^2 - 48 ~{\mathcal W}_+(g)
$$
for every Riemannian metric $g$. In the special case when $g$ is 
K\"ahler, we have 
$$
 {\mathcal W}_+(g) = \int_M \frac{s^2}{24} d\mu
$$
so this gives us 
$$
Y_{[g]}^2  \geq 96\pi^2  c_1^2 - 2\int_M s^2 d\mu  
$$
in the K\"ahler case. 
When $g$ is extremal, 
$$
\int_M s^2 d\mu = 32\pi^2 {\mathcal A}([\omega ])
$$
we therefore conclude that 
\begin{equation}
\label{yaya} 
Y_{[g]}^2  \geq  64\pi^2  \left(\frac{3}{2} c_1^2 - {\mathcal A}([\omega ]) \right) .
\end{equation}
Thus Yamabe constants are bounded away from zero in the interior 
of the controlled cone. 

Now Lemma \ref{rainmaker} tells us that any
 bilaterally symmetric extremal K\"ahler metric
 $g$  on $M=\CP_2 \# 2\overline{\CP}_2$
 has positive scalar curvature, and hence has positive Yamabe constant  
 $Y_{[g]}$. Since $c_1^2 (M) = 7$ and ${\mathcal A} (g) < 9$,
 inequality 
 (\ref{yaya}) therefore tells us that these metrics all satisfy 
 $$Y_{[g]} > 8\pi \sqrt{\frac{3}{2}(7)-9} = 4\pi \sqrt{6}.$$
With Lemma  \ref{zest} 
 and inequality (\ref{cushion}), this then tells us that the Sobolev constants 
 of these metrics 
 satisfy the uniform bound
 $$C_S < \frac{\max (6, 24\pi\sqrt{2})}{ 4\pi \sqrt{6}} = 2\sqrt{3}.$$
 But, by   previous work of the first and third authors \cite{chenweb}, 
 a uniform upper bound on Sobolev constants implies 
a weak compactness statement:

\begin{thm} \label{leap} 
Let  $g_i$ be an arbitrary sequence of unit-volume bilaterally symmetric extremal 
K\"ahler metrics on $M=\CP_2 \# 2\overline{\CP}_2$. 
Then there is a subsequence  $g_{i_\jmath}$ of these metrics which 
converges in the Gromov-Hausdorff topology to an extremal K\"ahler 
metric on a compact complex $2$-orbifold. 
\end{thm}

Now in general, the limit orbifold can certainly be different from 
$M$; in particular, the construction of Arezzo-Pacard-Singer \cite{arpasing} 
shows that when $x\to 0$, the limit orbifold is the manifold $\CP_1\times \CP_1$,
whereas when $x\to \infty$, the limit orbifold is the manifold $\CP_2$. 
In order to prove Theorem \ref{big}, what we therefore need to do is 
rule
out the bubbling off of curvature and topology in the case of increasing sequences of 
$x\in (0,x_0]$. 

\section{Formation of Bubbles}

Without further work, the results of \cite{chenweb} only allow us to 
conclude that sequences of extremal metrics with bounded Sobolev
constant have orbifold limits in the Gromov-Hausdorff sense.
However, just as in the earlier work of Anderson \cite{andprior,andsd,ander34} and Tian-Viaclovsky
\cite{tvale,tv1}, 
 the orbifold singularities can only arise by a  very specific 
mechanism of curvature concentration.

Suppose we have a sequence of unit-volume extremal K\"ahler metrics
$g_i$, and let us also assume that we have upper and lower bounds of
their scalar curvatures:
$$
|s_{g_i}|<{\rm const}.
$$
Assuming the uniform Sobolev constant bounds, the curvatures of
these metrics becomes unboundedly large at a point only if the $L^2$
norm of curvature reaches a definite threshold on arbitrarily small
balls, in the   precise sense that,  there are universal 
constants $C,\epsilon_0>0$ (depending only on the dimension
and the Sobolev constant) so that, if we set
$\varrho=C|\Riem_p|^{-2}$, we then have
\begin{equation}
\int_{B_\varrho (p)}|\Riem|^2d\mu \geq \epsilon_0. \label{conclusion}
\end{equation}
If there is no uniform bound for the sectional curvatures of the 
$g_i$, we can choose points $p_i$ centered at points of large curvature, 
rescale so that  $|\Riem_{p_i}|=1$, and then  take a pointed limit of some subsequence. The limit
(which is called a bubble) will then be a complete extremal K\"ahler
orbifold of total energy $\int|\Riem|^2d\mu \geq \epsilon_0$.
Because the metric was rescaled at each stage by  factors tending
toward infinity, the scalar curvatures are commensurately multiplied by  factors
tending toward zero, and our assumed  uniform bounds on $s$ then imply that  the
limit orbifold is actually  scalar-flat K\"ahler.

The bubble's structure at infinity is also known. The Sobolev
constant bound implies a global Euclidean volume growth lower bound
on the orbifold. The $\epsilon_0$-regularity theorem for extremal
metrics asserts  that 
$$
\sup_{B_{\varrho/2}(p)}|\Riem|\leq {C}\varrho^{-2}\left(\int_{B(\varrho,p)}|\Riem|^2\right)^{1/2}
~~\text{whenever} ~~\int_{B_\varrho (p)}|\Riem|^2\le\epsilon_0 . 
$$
Since the total energy is bounded, points far enough away from the
basepoint must be at the center of large balls of small energy, and
so that  $|\Riem|=o(\rad^{-2})$, where $\rad$ denotes the Euclidean radius
in asymptotic coordinates; 
moreover,  since the bubbles that actually concern us here will actually be
 anti-self-dual $4$-manifolds,
we can appeal to the results of 
Tian and Viaclovsky \cite[Proposition 5.2]{tv2} to obtain 
the faster 
curvature fall-off 
  $|\Riem|=o(\rad^{-4+\delta})$, for any $\delta > 0$,  at infinity.
 Results of Anderson \cite{andsd} and 
Tian-Viaclovski \cite{tv1} also show   that the bubble  has only finitely many ends,
each of which is asymptotically locally Euclidean (ALE), meaning
each end is asymptotic to the  the standard cone metric on
$S^3/\Gamma$. In our case one can improve this even further, as 
results of Li and Tam \cite[Theorem 4.1]{lit92} \cite[Theorem 1.9]{lit95}
then imply that  an ALE K\"ahler manifold has just 
one end.\footnote{In real dimension $4$, this  assertion can  be seen  more
directly by first showing that 
such a manifold can  be compactified  into a compact complex surface
by adding a divisor of positive self-intersection at each end.
But for any compact complex surface,  the intersection form on
 $H^{1,1}$ 
 is  always \cite{bpv} either Lorentzian or negative definite, so any 
 two divisors of positive self-intersection must necessarily meet. 
 The existence of two or more ends would  thus  lead to an 
 immediate contradiction.}

Now even these rescaled metrics may have curvature concentration points; 
however, we can then repeat the above procedure by picking points
$q_i$ a finite distance from $p_i$ that have
$|\Riem_{q_i}|\rightarrow\infty$, and rescaling and taking another
pointed limit. The  above-cited regularity 
 theorem from \cite{chenweb} implies that each stage
must have a ball of large radius that, after removing the points of
curvature concentration, has at least $\epsilon_0$ worth of energy.
Since there is a finite amount of total energy available, this
process must eventually terminate with a single-ended, scalar-flat
ALE K\"ahler manifold. Any blow-up limit that yields a smooth metric
will be called a {\sl deepest bubble}. Obviously, no curvature
can bubble off at all unless such a deepest bubble can be constructed. 
Indeed, the absence of a deepest bubble would imply a uniform bound on sectional curvature, 
 leading to smooth convergence everywhere.  

Let us now specialize  to the special case of a sequence $g_i$ of unit-volume 
bilaterally symmetric extremal K\"ahler metrics on $M=\CP_2\# 2 \overline{\CP}_2$.
First notice that Lemma \ref{zest} asserts  that such metrics do in fact satisfy 
 a uniform scalar curvature bound, so  the 
above discussion does indeed apply. 
Since the curvature of any deepest bubble $(X, g_\infty)$
necessarily arises from a concentration of the curvatures of the $g_i$, we thus  have
the following: 

\begin{lem}\label{bub} 
Let $g_i$ be a sequence of metrics as in Theorem \ref{leap}. If 
$g_i$ fails to converge modulo diffeomorphisms in the smooth topology, then 
there is a non-trivial asymptotically locally Euclidean (ALE) 
scalar-flat K\"ahler  manifold 
$(X, g_\infty)$
which  arises as a pointed
Gromov-Hausdorff limit of rescalings of a subsequence  of the $g_i$.
Moreover, the trace-free Ricci curvature and anti-self-dual Weyl curvature $W_+$ of
 $X$ necessarily satisfy 
 \begin{eqnarray*}
\int_X |\mathring{r}|^2 d\mu_{g_\infty} &\leq& \limsup_{i\to \infty}  \int_M |\mathring{r}|^2 d\mu_{g_i}\\
\int_X |W_-|^2 d\mu_{g_\infty} &\leq& \limsup_{i\to \infty}  \int_M |W_-|^2 d\mu_{g_i}
\end{eqnarray*}
\end{lem}

\bigskip 

Because \cite{lebicm} scalar-flat K\"ahler surfaces are anti-self-dual 
as oriented Riemannian $4$-manifolds, the  
 following regularity observation therefore applies to  the present context.

\begin{prop} \label{reg} 
Let $(X,g_\infty)$ be any ALE  anti-self-dual $4$-manifold, where
$$g_{\infty} = \text{Euclidean} + o (\rad^{-2+ \delta}),  ~~~~~ \partial g_{\infty} = o (\rad^{-3+ \delta}), $$
for some $\delta \in (0,1/2)$, where $\rad$ denotes the Euclidean radius. 
Consider the orbifold compactification $\hat{X}$ of $X$ obtained by 
adding an extra point at each end of $X$. Then $\hat{X}$ carries 
a canonical real-analytic structure such that  the conformal class 
$[g_\infty]$ extends to $\hat{X}$ as a real-analytic anti-self-dual 
conformal metric $[\hat{g}]$. 
\end{prop}
\begin{proof}
In inverted coordinates, the conformally related metric $\hat{g} = g_{\infty}\rad^{-4}$
is Euclidean $+ o(\varrho^{2-\delta})$, where $\varrho= 1/\rad$, so $g_{\infty}$
thus  determines a  
 $C^{1,\alpha}$ conformal metric $\hat{g}_0$ 
on the orbifold compactification $\hat{X}$ of $X$, for any
 $\alpha \in (1/2,  1 - \delta)$.  Let $Y\approx \RR^4$
 be a uniformizing chart for any end of $X$. 
  The Christoffel symbols of $\hat{g}_0$ 
 are thus of class $C^{0,\alpha}$, and the standard Atiyah-Hitchin-Singer 
 formulation of the twistor construction 
 \cite{AHS,lebicm} 
 \begin{center}
\mbox{
\beginpicture
\setplotarea x from 0 to 350, y from -10 to 155
\putrectangle corners at 75 130 and 300 70
\circulararc 360 degrees from 30 135 center at 30 105
\circulararc -70 degrees from 45 106  center at 37 106
\ellipticalarc axes ratio 3:1 -180 degrees from 60 105
center at 30 105
\ellipticalarc axes ratio 2:1 60 degrees from 194 25
center at 187 25
\ellipticalarc axes ratio 4:1 65 degrees from 193 106
center at 187 106
\arrow <2pt> [.1,.3] from 187 25 to 197 23
\arrow <2pt> [.1,.3] from 188 106 to 197 105
\arrow <2pt> [.1,.3] from 187 25 to 189 29
\arrow <2pt> [.1,.3] from 187 106 to 189 109
\arrow <2pt> [.1,.3] from 37 106  to 39 94
\arrow <2pt> [.1,.3] from 37 106  to 46 108
\put {\circle*{3}} [B1] at 188 25
\put {\circle*{3}} [B1] at 188 106
\put {\circle*{3}} [B1] at 38 106
\put {$S(\Lambda^+)$} [B1] at 320 100
\put {$Y$} [B1] at 320 10
\put {$\downarrow$} [B1] at 187 55
%\put {$\wp$} [B1] at 195 55
{\setlinear 
\plot  75 130 95 150  /
\plot 300 130 280 150 /
\plot  95 150 280 150  /
\plot 187 138 187 80 /
\plot 170 103 204 103 /
\plot 173 109 201 109 /
\plot 170 103 173 109 /
\plot 201 109 204 103 /
\plot 170 19 204 19 /
\plot 173 31 201 31 /
\plot 170 19 173 31 /
\plot 201 31 204 19 /
}
{\setquadratic
\plot  75 -10 85 21 95 30 /
\plot  75 -10 187 5 300 -10 /
\plot 300 -10 290 21 280 30 /
\plot 95 30 187 40 280 30 /
}
\endpicture
}
\end{center}
gives us an almost-complex structure $J$ on the $6$-manifold $Z$
defined as the $2$-sphere bundle $S(\Lambda^+)\to Y$. We now 
apply the Hill-Taylor version \cite{hiltay} of the Newlander-Nirenberg theorem 
for rough almost-complex structures. Since $J$ is an almost-complex 
structure of class $C^{0,\alpha}$, $\alpha > 1/2$,  its Nijenhuis tensor ${\mathcal N}_J$ is
not only well-defined in the distributional sense, but  actually  
\cite[Lemma 1.2, Remark 1.5]{hiltay}
of  Sobolev regularity better than $L^2_{-1/2}$. But since $[\hat{g}_0] = [g_\infty ]$
is anti-self-dual away from the added point at infinity, ${\mathcal N}_J$ is supported
at the added twistor fiber, which is a submanifold of $Z$. However, 
a   distribution of Sobolev class $L^2_{-1/2}$ that is supported on a real hypersurface
is automatically  zero, as follows from  \cite[Chapter 1, Theorem 11.1]{lima} and duality. 
Thus it follows  that ${\mathcal N}_J=0$ in the distributional sense, and 
$(Z,J)$ is therefore \cite[Theorem 1.1]{hiltay} a complex manifold. 
Inverting the twistor correspondence \cite{hitka,lebtour,pnlg} therefore realizes $Y$ as the
real slice of the moduli space of rational curves $\CP_1 \subset Z$ of
with normal  bundle ${\mathcal O} (1) \oplus {\mathcal O} (1)$; and, more importantly,
this moduli space carries
an anti-holomorphic involution fixing $Y$  and 
a natural holomorphic conformal structure   whose restriction to $Y$ is $[\hat{g}_0]$. 
In particular, $Y$ can be given a real-analytic structure in which $[\hat{g}_0]$
is represented by a real-analytic metric $\hat{g}$. Dividing $Y$ by the appropriate
group $\Gamma$ 
now gives us an orbifold chart associated with the given end, 
and repeating the same argument for each end then gives 
$\hat{X}$ the promised structure. 
\end{proof}

Notice that the results of Tian and Viaclovsky \cite{tv2}   tell us that 
the hypotheses of this proposition hold whenever an anti-self-dual 
manifold $(X,g_\infty)$ arises as a bubble. 
This has many useful consequences: 

\begin{prop} \label{neg} 
Let $(X,g_\infty)$ be any scalar-flat ALE  anti-self-dual $4$-manifold, where
$$g_{\infty} = \text{Euclidean} + o (\rad^{-2+ \delta}),  ~~~~~ \partial g_{\infty} = o (\rad^{-3+ \delta}), $$
for some $\delta \in (0,1/2)$, where $\rad$ denotes the Euclidean radius. 
Then $X$ has negative intersection form.
 Moreover, after possibly passing to better
charts at infinity, $g_\infty$ actually satisfies the improved fall-off conditions
$$g_{\infty} = \text{Euclidean} + O (\rad^{-2}),  ~~~~~ \partial^m g_{\infty} = 
O (\rad^{-2-m}) ~\forall  m\geq 1.$$
In particular, these conclusions apply to any deepest bubble $(X,g_\infty)$ arising 
as in Lemma \ref{bub}. \end{prop}

\begin{proof}
The real-analytic conformal class $[\hat{g}]$ can be represented by a
real-analytic metric $\hat{g}$ whose scalar-curvature $s$ doesn't change sign; 
for example, such a metric can be constructed via Trudinger's trick of 
rescaling by the lowest eigenfunction of the Yamabe Laplacian. 
Setting $g_\infty = u^2 \hat{g}$ for $u> 0$, one then finds that
$u$ is real-analytic and proper on $\hat{X}-X$, 
and solves the equation
 \begin{equation} \label{green} 
0 = (\Delta + \frac{s}{6}) u,
\end{equation}
so that we must have $s > 0$ by examination of the minima of $u$. 

Let us now represent the deRham groups $H^2(\hat{X})$ by harmonic $2$-forms 
with respect to $\hat{g}$. Letting  $\varphi$ be any such harmonic $2$-form, 
its  self-dual part $\varphi^+ = (\varphi + \star \varphi )/2$ then satisfies the  
Weitzenb\"ock formula 
$$
0 = (d+d^*)^2 \varphi^+ = \nabla^*\nabla \varphi^+ - 2 W(\varphi^+ , \cdot ) + \frac{s}{3} \varphi^+ =
\nabla^*\nabla \varphi^+ +\frac{s}{3} \varphi^+ . 
$$
Taking the inner product with $\varphi^+$ and integrating, we thus conclude that 
$\varphi^+$ vanishes. Hence any harmonic form on $\hat{X}$ is anti-self-dual,
and the intersection form of $\hat{X}$ is negative. Since $H^2(X) = H^2_c(X)=H^2(\hat{X})$,
the same therefore applies to our original manifold $X$. 

Now (\ref{green}) reveals that $u$ is in fact a linear superposition
of the Yamabe Green's functions of the ends. Since the Yamabe Green's function of 
an anti-self-dual manifold is real-analytic in any real-analytic conformal gauge, and
has a local expansion \cite{atgrn} 
$$G_y = \frac{1}{4\pi^2 \varrho^2} + \text{bounded} $$
without $\log (\varrho )$ term,  the improved asymptotic expansion for $g_\infty$ now
follows by inverting geodesic coordinates about each  point of 
$\hat{X}-X$. 
\end{proof}

This said, we now immediately have the following: 

\begin{lem}\label{bub2}
Let $g_i$ and $(X, g_\infty)$ be as in Lemma \ref{bub}. Then 
$X$ is  diffeomorphic to an open subset of $M$. Moreover, 
$b_1(X) = b_3 (X) =0$,  and $b_2 (X) \leq 2$.
\end{lem} 
\begin{proof}
The bubble $(X, g_\infty)$ is obtained as a pointed Gromov-Hausdorff
limit of rescaled versions of small metric balls in $M$, and the rescaling
is done in such a manner as to arrange that the sectional curvatures
are bounded. One therefore gets smooth convergence  on compact
subsets by passing to a subsequence and applying suitable diffeomorphisms. 
But $X$ is diffeomorphic to the interior of a compact domain  $U\subset X$ 
with smooth boundary $S^3/\Gamma$. This  domain 
 can then be mapped diffeomorphically into the manifold,  resulting in 
a decomposition  
\begin{equation}
\label{pieces} 
M\approx U\cup_{S^3/\Gamma}V
\end{equation}
where $U$ and $V$ are manifolds-with-boundary,
$X\approx \Int (U)$, 
 and  
$\partial U=\partial V=  S^3/\Gamma$.
Since  $M= \CP_2 \# 2 \overline{\CP}_2$ is simply connected, 
  the Mayer-Vietoris sequence tells us that 
 both $U$ and $V$ have $b_1=b_3=0$, while 
 $$
 H^2(M, \RR ) = H^2(U, \RR) \oplus H^2 (V, \RR) . 
 $$
Since the analogous statements similarly hold for  homology, 
the intersection form of $X\approx \Int (U)$ is just the restriction of 
the intersection form of $H^2 (M)$ to the linear subspace $H^2(U)\subset H^2 (M)$. 
But the intersection form of $X$ is negative by Proposition \ref{neg},
and it thus follows that $b_2 (X) \leq b_-(M) = 2$. 
 \end{proof}

The following will also prove quite useful.

\begin{lem} \label{puck} 
Let $g_i$ and $(X, g_\infty)$ be as in Lemma \ref{bub}.
If the open subset of Lemma \ref{bub2} cannot be taken to be invariant 
under  under $F_1\leftrightarrow F_2$,
then curvature is accumulating in more than one region, and 
\begin{eqnarray*}
2\int_X |\mathring{r}|^2 d\mu_{g_\infty} &\leq&\limsup_{i\to \infty}  \int_M |\mathring{r}|^2 d\mu_{g_i}\\
2\int_X |W_-|^2 d\mu_{g_\infty} &\leq& \limsup_{i\to \infty}  \int_M |W_-|^2 d\mu_{g_i}.
\end{eqnarray*}
\end{lem}
\begin{proof}
If we go far out in the sequence, the deepest bubble essentially arises by rescaling 
the interior of a  domain $U_j$ of small diameter, where the $L^2$ norms of curvatures
on  $U_j$ closely approximate the corresponding norms for the deepest bubble $X$. 
Now move this domain by the isometry  $F_1\leftrightarrow F_2$
to obtain another domain $U_j^\prime$. If, after again passing to a subsequence, the 
rescaled distance from $U_j$ to $U_j^\prime$ remains bounded, the pointed 
Gromov-Hausdorff limit will include the limits of the $U_j^\prime$, and 
we will have an induced isometry which exchanges the two. Otherwise, 
the concentration of curvature represented by the $U_j \cup U_j^\prime$ is
 reflected, not by $(X, g_\infty)$, but rather by disjoint  two copies of it, and we 
 therefore get a factor of two in the relevant curvature inequalities. 
\end{proof}

\begin{prop}\label{toric} 
Let  $(X, g_\infty)$ be a deepest bubble, as in Lemma \ref{bub}. 
Then $g_\infty$ is toric,  and $H_2(X)$ is generated
by holomorphically embedded $\CP_1$'s in $X$.  
\end{prop}
\begin{proof}
Recall that a Killing vector field $\eta$ on a  Riemannian manifold is completely
determined by its $1$-jet at any point $p$, since the restriction of such 
a field to any geodesic solves a second-order ODE---namely, Jacobi's equation. 
Because Killing's equation says $\nabla_a\eta_b =-\nabla_b\eta_a$, the initial data for
a Killing field may be identified with $\Lambda^1_p\oplus \Lambda^2_p$.
If we equip the bundle $\Lambda^1\oplus \Lambda^2$ with a connection
defined by 
$$D_a (\eta_b , \varphi_{cd}) : = (\nabla_a\eta_b- \varphi_{ab}, 
\nabla_a\varphi_{cd} -
{{\mathcal R}^e}_{acd}\eta_e),$$
we thus conclude that 
Killing fields  precisely correspond, 
via  $\eta^a\mapsto(\eta_b , \nabla_c \eta_d )$, 
to $D$-parallel sections of $\Lambda^1\oplus \Lambda^2$. 
A constant rescaling  $g\rightsquigarrow  cg$ of the metric 
merely induces a homothety $(\eta , \varphi ) \mapsto  (c\eta, c\varphi )$, and
so in particular does not affect the correspondence between 
$1$-parameter subgroups of the isometry group and parallel
line-sub-bundles of $\Lambda^1\oplus \Lambda^2$.

Since $(X, g_\infty)$ is constructed as a pointed limit, it comes equipped
with a base-point $p\in X$ which,  after passing to 
a suitable subsequence,  can be thought of  as the limit of a sequence of points
$p_i\in M$ associated with rescalings $c_ig_i$ of the given metrics. 
Since each of the metrics $g_i$ is toric, the generators of the torus
action span a $2$-plane $\Pi_i\subset \Lambda^1_{p_i}\oplus \Lambda^2_{p_i}$,
and since, by construction,  $T_pX$ is canonically identified with each $T_{p_i}M$,
we therefore obtain a sequence in the Grassmannian $Gr_2(\Lambda^1_p\oplus \Lambda^2_p)$.
But  the latter is compact, so we may arrange, by again passing to  
a subsequence, that $\Pi_i \to \Pi$ for some $2$-plane $\Pi\subset \Lambda^1_p\oplus \Lambda^2_p$.
Since the pull-backs of the rescaled metrics $c_ig_i$ via suitable diffeomorphisms  converge
in $C^\infty$ to $g_\infty$ on compact subsets of $X$, the holonomy transformation
of $\Lambda^1_p\oplus \Lambda^2_p$ induced by $D$-parallel transport 
in $\Lambda^1\oplus \Lambda^2\to X$ around
any smooth loop in $\gamma$ based at $p$ is always the limit of the holonomy transformations
in $(M, c_ig_i)$ around loops based at $p_i$, and since the latter holonomy
transformation act trivially on $\Pi_i$, it follows that $\Pi\subset \Lambda^1_p\oplus \Lambda^2_p$
is invariant under the holonomy transformations induced by $D$-parallel transport.
Hence $\Pi$ extends uniquely via $D$-parallel transport to a $D$-parallel
sub-bundle of $\Lambda^1\oplus \Lambda^2\to X$, and we thereby obtain 
two non-proportional Killing fields on $(X,g_\infty )$. Moreover, these Killing
fields can now be seen to arise, under suitable  diffeomorphisms,  as smooth limits 
on compact sets of 
linear combinations of the original two commuting Killing fields, so  the two Killing fields
$\eta$, $\tilde{\eta}$  we obtain in this way on $X$ automatically commute with each other.

Now the two Killing fields $\eta$ and $\tilde{\eta}$ become conformal 
Killing fields on the anti-self-dual orbifold $\hat{X}= X\cup \{ \infty\}$ which 
vanish at the added orbifold point. But Pontecorvo \cite{max} has pointed
out that, even locally,  any conformally flat scalar-flat K\"ahler surface is locally
symmetric, and it therefore follows that a non-flat ALE scalar-flat K\"ahler  
surface like $(X, g_\infty )$ can  never be conformally flat.
The compact orbifold  $\hat{X}= X\cup \{ \infty\}$ therefore has 
 $W_-\not\equiv 0$. Thus the usual proofs of the 
Ferrand-Lelong/Obata theorem \cite{lelongferrand,obata} apply in this orbifold setting, and 
the conformal group of $(\hat{X}, [\hat{g}])$ coincides with the isometry group 
of some orbifold metric in the conformal class.   Since our conformal Killing fields 
$\eta$ and $\hat{\eta}$  fix the orbifold point $\infty\in \hat{X}$, 
their action on $\hat{X}$ is therefore completely determined
(via the exponential map)  by their action on 
the tangent space at $\infty$ in a local uniformizing chart. 
Up to finite covers, this therefore gives  us a faithful  $SO(4)$-valued 
representation of the group  generated by $\eta$ and $\tilde{\eta}$.
But $\eta$ and $\tilde{\eta}$ are independent, and  $[\eta , \tilde{\eta}]=0$. 
The $2$-dimensional Abelian Lie group they generate must therefore be covered by a
maximal torus in SO(4), and so must be compact. Hence the Killing
fields $\eta$ and $\hat{\eta}$ generate an action of the compact group
$T^2$  on $(X, g_\infty )$.  Moreover, by replacing $\eta$ and $\hat{\eta}$
with linear combinations, we may find an asymptotic chart for
$X$ in which 
\begin{eqnarray*}
\eta &=& x^1 \frac{\partial}{\partial x^2}  - x^2 \frac{\partial}{\partial x^1}   + O(1) 
\\
\tilde{\eta} &=& x^3 \frac{\partial}{\partial x^4}  - x^4 \frac{\partial}{\partial x^3}   + O(1) 
\\
g_\infty&=& \text{Euclidean} + O(\rad^{-2})\\
\partial g_\infty
&=& O(\rad^{-3})
\end{eqnarray*}

Now the original Killing fields on $(M, c_ig_i)$ were real-holomorphic, so their
$1$-jets in fact all belonged  to $\Lambda^1_{p_i}\oplus \Lambda^{1,1}_{p_i}$.
It follows that the limit plane $\Pi$ is therefore a sub-space of  
$\Lambda^1_{p}\oplus \Lambda^{1,1}_{p}$,
and  $\eta$ and  $\tilde{\eta}$ are therefore also real holomorphic.
Since these fields preserve both the metric $g_\infty$ and the limit complex structure
$J=J_\infty$, they therefore preserve the limit K\"ahler form $\omega= \omega_\infty$, too.
We can thus arrange that in our asymptotic chart  we also have 
$$
\omega = dx^1\wedge dx^2 + dx^3 \wedge dx^4  + O(\rad^{-2}) ~. 
$$

  By Lemma \ref{bub2}, $b_1(X)=0$, so both $\eta$ and $\tilde{\eta}$
are  globally Hamiltonian; that is, there exist smooth functions 
$t_1, t_2: X\to \RR$ such that 
$$\omega (\eta , \cdot )  = - d t_1, ~~ \omega (\tilde{\eta}  ,\cdot )  = - d   t_2~,$$
and the above asymptotics  therefore give us 
\begin{eqnarray*}
t_1 &=& 
 \frac{|x^1|^2 + |x^2|^2}{2} 
  + O(\rad) \\
t_2  &=&  \frac{|x^3|^2 + |x^4|^2}{2} 
  + O(\rad)
\end{eqnarray*}
This shows  that $t_1 + t_2 > \rad^2/3$ on the complement of a compact set, 
and it therefore follows  that the moment map 
$$
\vec{t}= (t_1 , t_2) : X \longrightarrow  \RR^2$$is proper. Moreover, 
any linear combination $a_1t_1 + a_2 t_2$, where $a_1$ and $a_2$ are
positive constants, is proper by the same argument. 
Since $\eta$ and $\tilde{\eta}$ are Killing fields on our K\"ahler manifold, a generic 
such linear combination has only non-degenerate critical points, and  
is therefore a  Morse function. 
The essence of \cite{atcvx} therefore applies, despite our non-compact setting. 
Namely, the image of $\vec{t}$ is a convex subset of the plane, 
bounded by two half-lines and a finite number of line segments
of rational slope: 
\begin{center}
\begin{picture}(240,90)(0,10)
\put(80,50){\vector(0,1){50}}
\put(110,20){\vector(1,0){50}}
\put(110,20){\line(-2,1){20}}
\put(80,50){\line(1,-2){10}}
\end{picture}
\end{center}
The two boundary rays arise from the fixed-point sets of $\eta$ and $\tilde{\eta}$,
which are totally geodesic copies of $\CC$ emanating from two 
fixed points of the torus action. 
The boundary segments arise from other sets where some circle subgroup of
$T^2$ acts trivially. The inverse image of each such segment is a
totally geodesic surface,  which  must in turn be a
topological $2$-sphere because of  the residual circle action; moreover, 
each such $2$-sphere is the zero locus of a 
holomorphic vector field,
and so is a  holomorphic curve. Finally, the 
union of these $\CP_1$'s  is a deformation retraction of
$X$, as may be accomplished by  pushing  along  the flow of 
some  Morse function $a_1t_1 + a_2 t_2$. Thus $b_2(X)$ is exactly the 
number of line segments, and $H_2(X)$ is  generated by holomorphically
embedded $\CP_1$'s, as promised. 
\end{proof}

Given the amount of structure we have already 
 displayed, it seems extremely plausible  that
  our
toric manifolds  
$(X,g_\infty)$ actually  number among    the scalar-flat K\"ahler instantons 
explicitly constructed by Calderbank and Singer \cite{caldsing}. The latter arise via 
a special form of the  Joyce ansatz, and  one of Joyce's results 
 \cite[Theorem 2.4.5]{rejoyce} in any case 
implies  that our metrics $g_\infty$  are at least 
locally expressible in his framework. Moreover, the results of 
Fujiki \cite{fujiki}, although not immediately  applicable here, 
 make it seem very  likely that a
 global result along these lines should actually hold.  

For our purposes, however, it will not actually be necessary to 
know the possible bubble metrics $g_\infty$ in closed form. Instead, the 
next few lemmas will supply all the information we need.

\begin{lem} \label{bottom} 
Let  $X$ be as in Lemma \ref{bub}. 
Then $b_2(X) \neq 0$.  \end{lem}
\begin{proof}
If $b_2(X)=0$, the proof of Proposition \ref{toric} shows that there is 
a Morse function on $X$ with exactly one critical point. 
\begin{center}
\begin{picture}(240,70)(0,10)
\put(80,20){\vector(0,1){50}}
\put(80,20){\vector(1,0){50}}
\put(100,40){\vector(1,1){20}}
\end{picture}
\end{center}
Thus $X$ is 
diffeomorphic to $\RR^4$. Hence $\hat{X}=X\cup\{ \infty\}$ is diffeomorphic 
to $S^4$, and the signature formula then shows that $\hat{X}$ is conformally flat. 
Thus $(X,g_\infty )$ is a conformally flat scalar-flat K\"ahler manifold, and 
Pontecorvo's theorem \cite{max} therefore tells us that it is locally symmetric.
Its curvature 
fall-off at infinity 
 therefore forces  $(X,g_\infty )$ to be flat\footnote{This could also be proved
by instead using a result of 
Anderson 
 \cite[Corollary 3.9]{andsd}.}. 
 But a deepest bubble $X$ cannot be flat, so  this case simply never
arises. 
\end{proof}

While this result may look innocuous, it is actually  heavily dependent
on the fact that $X$ is known to be toric. For example,  there is a
non-trivial Ricci-flat ALE metric (with isometry group $SO(3)$) 
on $\CP_2$ minus a smooth a conic; but  this space is 
diffeomorphic to $T\RR {\mathbb P}^2$, and so has 
$b_2=0$.
 Perhaps the most dramatic 
consequence of the toric condition is that it guarantees the 
existence of a Morse function whose critical points all have 
even index; thus  $X$ must, in particular, be 
simply connected. We leave it as  an exercise 
 to check that, more generally,  any  {\em simply connected} 
scalar-flat K\"ahler 
bubble  must have $b_2\neq 0$.
The point is that 
$X$ must then  either  be hyper-K\"ahler, in which 
case one can appeal to the results of 
  Kronheimer \cite{krontor}, or else the Ricci form will correspond to  a non-trivial 
  bounded harmonic $2$-form on  $\hat{X}$.

\begin{lem}\label{b21}
 Let  $(X, g_\infty)$ be a deepest bubble, as in Lemma \ref{bub}.
If $b_2 (X)= 1$, then $X$ is diffeomorphic to a complex line bundle of
negative degree over  $\CP_1$, in such a manner that the zero section 
 corresponds to a holomorphic curve  $\CP_1\subset X$.  In particular, 
the intersection form of $X$ is $(-k)$ for some integer $k\geq 1$, and the 
group $\Gamma$ at infinity is the cyclic group $\ZZ_k$. 
\end{lem}
\begin{proof}
By flowing along the gradient of a suitable Morse function $a_1t_1 + a_2 t_2$, 
the proof of Proposition \ref{toric} shows that $X$ is diffeomorphic to a
tubular neighborhood of a single holomorphically embedded $\CP_1$: 
\begin{center}
\begin{picture}(240,90)(0,10)
\put(80,50){\vector(0,1){50}}
\put(110,20){\vector(1,0){60}}
\put(110,50){\vector(1,1){30}}
\put(110,20){\line(-1,1){30}}
\end{picture}
\end{center}
Since $X$ has negative intersection form by Proposition \ref{neg}, 
the claim follows. 
\end{proof}

\begin{lem} \label{b22} 
Let  $(X, g_\infty)$ be a deepest bubble, as in Lemma \ref{bub}.
If $b_2 (X)= 2$, then $X$ is diffeomorphic to the $4$-manifold obtained by 
plumbing together two complex line bundles  over $\CP_1$.
Moreover, there is a basis for $H_2 (X,\ZZ)$,  represented by a pair of totally 
geodesics and 
holomorphic $\CP_1$'s, in which the intersection form becomes 
$$\left(\begin{array}{cc}-k & 1 \\1 & -\ell\end{array}\right)$$
for some positive integers $k\geq 2$ and  $\ell\geq 1$. 
Finally, the 
group $\Gamma$ at infinity is the cyclic group $\ZZ_{k\ell -1}$.
\end{lem}
\begin{proof}
By following the gradient lines of a suitable Morse function $a_1t_1 + a_2 t_2$, 
the proof of Proposition \ref{toric} shows that $X$ is diffeomorphic to a
neighborhood of a pair of  holomorphically embedded $\CP_1$'s: 
\begin{center}
\begin{picture}(240,90)(0,10)
\put(80,50){\vector(0,1){50}}
\put(110,20){\vector(1,0){50}}
\put(110,20){\line(-2,1){20}}
\put(80,50){\line(1,-2){10}}
\put(110,50){\vector(1,1){30}}
\end{picture}
\end{center}
Since this neighborhood can be obtained by plumbing together 
the normal bundles of these $\CP_1$'s,  the 
intersection form certainly is given by a matrix of the 
above form for some  $k, \ell \in \ZZ$. 
But $X$ has negative intersection form by Proposition 
\ref{neg}, so we must have $-k< 0$, $-\ell < 0$, and $-k - \ell + 2 < 0$,
as may be seen by taking the self-intersections of the two generators and 
their sum. Thus  $k$ and $\ell$ are both  positive, 
and one of them (without loss of generality, 
$k$) must be large than $1$.

Now the $3$-manifold $S^3/\Gamma$ must be diffeomorphic to any 
level set of the Morse function occurring above the highest critical point, and 
our plumbing picture says that this $3$-manifold can therefore be 
obtained from the disjoint union of $S^3/\ZZ_k$ and $S^3/\ZZ_\ell$, thought of 
as principal circle bundles over the two $\CP_1$'s, by deleting
a trivialized neighborhood of a  fiber in each and then 
identifying the resulting  $S^1\times S^1$ boundaries via an interchange of the factors. 
It follows that $S^3/\Gamma$ can be constructed 
by  gluing together two solid tori $S^1\times D^2$ along their
boundaries $S^1\times S^1 = \RR^2/\ZZ^2$ via 
$$
\left(\begin{array}{cc}-k & 1 \\1 & 0\end{array}\right)^{-1}
\left(\begin{array}{cc}0 & 1 \\1 & 0\end{array}\right)
\left(\begin{array}{cc}-\ell & 1 \\1 & 0\end{array}\right)
=\left(\begin{array}{cc}-\ell & 1 \\ 1-k\ell  & k\end{array}\right)\in GL(2, \ZZ) ~.
$$
Thus a meridian of one torus becomes
a circle of slope $\ell/(k\ell -1)$ on the other, 
and the $3$-manifold at infinity is therefore a Lens space $L(k\ell -1, \ell)$,   with 
fundamental group $\Gamma \cong \ZZ_{k\ell -1}$. 
\end{proof} 

\section{Obstructions to Bubbling}

In light of the information gleaned from \S \ref{action}, the 
curvature of bilaterally symmetric extremal K\"ahler metrics 
on  $M=\CP_2 \# 2\overline{\CP}_2$ is rather tightly constrained, 
at least when $x\in (0,x_0]$; indeed, such metrics have
$$
\frac{1}{32\pi^2}\int s^2 d\mu = {\mathcal A} ([\omega ]) < 8
$$
if 
 $x= \varepsilon / \beta$ 
is  in this range. 
However, the
K\"ahler condition  implies that 
$$
\frac{1}{8\pi^2} \int_M \left(\frac{s^2}{4}  - |\mathring{r}|^2\right)d\mu = c_1^2 (M) = 7,
$$
so 
any of these extremal K\"ahler  metrics actually has 
$$
\int_M |\mathring{r}|^2~d\mu  <  8\pi^2 . 
$$
By the signature formula, we also have
$$\int_M |W_-|^2 d\mu = - 12\pi^2 \tau (M) + \int_M |W_+|^2 d\mu = 12\pi^2 
+ \int_M \frac{s^2}{24} d\mu$$
for any K\"ahler metric. Thus  extremal K\"ahler metrics with 
${\mathcal A}([\omega ]) < 8$ also satisfy 
\begin{equation}
\label{lid}
\frac{1}{4\pi^2}\int_M |W_-|^2 d\mu < 3 + \frac{8\cdot 32\pi^2}{24\cdot 4\pi^2}= \frac{17}{3}
\end{equation}
implying, in particular, that 
$$
\int_M |W_-|^2 d\mu < 23\pi^2.
$$
We will now use this and similar knowledge to prove, in stages, that curvature bubbling
does not occur  for 
sequences of such metrics.

One of the tools we will use repeatedly is a variant of the Gauss-Bonnet formula. 
If $(X,g_\infty)$ is any ALE $4$-manifold with group $\Gamma$ at infinity, 
then the corrected form\footnote{This may be proved by Chern's method \cite{chern}; namely, if we choose an asymptotically
outward pointing vector field, the Gauss-Bonnet integral counts the number of 
zeroes of the vector field, plus a boundary integral which would contribute  
$-1$ in Euclidean space.}
 of the Gauss-Bonnet formula reads 
$$
\frac{1}{8\pi^2} \int_X \left(\frac{s^2}{24}+|W_+|^2 +|W_-|^2 - \frac{|\mathring{r}|^2}{2}
\right) d\mu_{g_\infty} = \chi (X)-\frac{1}{|\Gamma |}$$
where $\chi$ is the topological Euler characteristic of the 
non-compact manifold, and $|\Gamma|$ is the order of the group. 
When $(X,g_\infty)$ is scalar-flat K\"ahler, this simplifies to become
\begin{equation}
\label{gammab}
\frac{1}{8\pi^2} \int_X \left(|W_-|^2 - \frac{|\mathring{r}|^2}{2}
\right) d\mu_{g_\infty}  = \chi (X)-\frac{1}{|\Gamma |}~.
\end{equation}
Our first key observation  is that our deepest bubbles must
necessarily have $\Gamma \neq \{ 1\}$. 

\begin{lem} \label{nae} 
Let $g_i$ be a sequence of  unit volume bilaterally symmetric 
extremal K\"ahler metrics
on $M= \CP_2\# 2 \overline{\CP}_2$ with ${\mathcal A}(g_i)< 8-\delta$ for some 
$\delta > 0$, and suppose that  sectional curvatures
are not 
uniformly bounded
for  this sequence.  Let $(X,g_\infty)$ be a deepest bubble extracted  by rescaling
a subsequence  at points of maximal curvature. Then, 
 at infinity, 
the ALE scalar-flat
K\"ahler manifold  $(X,g_\infty)$ is asymptotic  to $\RR^4/\Gamma$
for some $\Gamma \neq \{ 1\}$. That is,  $(X,g_\infty)$ cannot be strictly
asymptotically Euclidean. 
\end{lem}
\begin{proof}
Suppose we had such a bubble with $\Gamma =1$. Then, by Propositions \ref{reg} and \ref{neg}, 
 the one-point compactification  $\hat{X}= X\cup  \{ \infty\}$
of $X$ is a compact anti-self-dual $4$-manifold
with negative intersection form. The signature formula 
$$
\tau (\hat{X}) = \frac{1}{12\pi^2}\int_{\hat{X}} \left(
 |W_+|^2-|W_-|^2  \right)d\mu_{\hat{g}}
$$
and the conformal invariance of ${\mathcal W}$ therefore give us 
$$
\int_{X}
|W_-|^2  d\mu_{{g}_\infty} = 12\pi^2 b_2 (X).
$$
On the other hand, Lemma \ref{bub2} tells us that 
$b_1=b_3=0$, so (\ref{gammab}) with $|\Gamma |=1$ becomes 
$$
 \int_X \left(|W_-|^2 - \frac{|\mathring{r}|^2}{2}
\right) d\mu_{g_\infty} = 8\pi^2 b_2 (X)
$$
and we therefore conclude that 
$$
 \int_X  |\mathring{r}|^2 d\mu_{g_\infty} = 8\pi^2 b_2 (X). 
$$
But  our assumptions imply that 
  $$
  \limsup_{i\to\infty} \int_M {|\mathring{r}|^2} d\mu_{g_i} \leq  8\pi^2(1-\delta),$$
  and Lemma \ref{bub} tells us   that  
  $$  \int_X {|\mathring{r}|^2} d\mu_{g_\infty} \leq \limsup_{i\to\infty}
   \int_M {|\mathring{r}|^2} d\mu_{g_i}
  < 8\pi^2$$
 so  we must have $b_2(X)=0$. But  this now implies that both $\mathring{r}$
 and $W_-$  vanish identically, forcing $g_\infty$ to  be flat. However, this is impossible,
 since $(X, g_\infty)$ is  a deepest bubble.   Thus deepest
 bubbles with $\Gamma =\{ 1\}$ cannot arise in the present context. 
 \end{proof}

%We remark, in passing,  that the Ricci-flat case in the above proof could 
%have been  excluded in 
%many other 
%ways here. For example, the Ricci-flat condition on an ALE metric
%forces its  mass to vanish \cite{krontor}, and $\Gamma = \{ 1\}$ can therefore  
%be excluded by invoking the positive mass theorem  \cite{syaction}. 

Similar reasoning gives us: 

\begin{lem} \label{swap} 
Let $g_i$ and $(X, g_\infty)$ be as in Lemma \ref{nae}. 
Then $X$ is diffeomorphic to a region of $M$
which is invariant under $F_1\leftrightarrow F_2$, and this 
$\ZZ_2$-action  induces  a holomorphic  isometric involution 
of $(X, g_\infty)$. 
\end{lem} 
\begin{proof} 
If $(X, g_\infty)$ is a deepest bubble arising as in Lemma \ref{nae},
we now know that  $\Gamma \neq \{ 1\}$, and hence 
$|\Gamma|\geq 2$. Moreover, $b_1(X)=b_3(X)=0$
by Lemma \ref{bub2} and $b_2\neq 0$ by Lemma \ref{bottom}.
Hence $\chi (X) \geq 2$, and 
$$
\int_X |W_-|^2d\mu \geq 8\pi^2 (2-\frac{1}{2}) \geq 12\pi^2. 
$$
 Since inequality (\ref{lid}) tells us that 
we have $< 23\pi^2$ worth of $\|W_-\|^2$ to bubble away,
Lemma \ref{puck} therefore 
shows
that   $F_1\leftrightarrow F_2$ must induce an 
isometry of $X$, and that $X$ is actually diffeomorphic to a region of $M$
which is invariant under the corresponding $\ZZ_2$-action. 
\end{proof} 

 \begin{lem} \label{elf} 
 Let $g_i$ and $(X, g_\infty)$ be as in Lemma \ref{nae}. 
 If $b_2(X)=2$, then $\Gamma \cong \ZZ_3$, and 
$X$ has intersection form 
  $$\left(\begin{array}{rr}-2 & 1 \\1 & -2\end{array}\right).$$
  \end{lem} 
  \begin{proof}
  In conjunction with inequality (\ref{lid}), 
  Lemma \ref{bub} 
tells us that 
$$
\frac{17}{6}  >  \frac{1}{8\pi^2} \int_M |W_-|^2d\mu \geq 1+b_2(X)-\frac{1}{|\Gamma |} 
$$
so that 
$$
\frac{11}{6} + \frac{1}{|\Gamma |}  >  b_2 (X). 
$$
When $b_2(X) =2$, we thus have $|\Gamma |\leq 5$. 
  
But when $b_2(X)=2$, Lemma \ref{b22} tells us that 
the intersection form  is
$$\left(\begin{array}{cc}-k & 1 \\1 & -\ell\end{array}\right)$$
for some $k\geq 2$,  $\ell \geq 1$, and that $\Gamma = \ZZ_{k\ell -1}$. 
But   Lemma \ref{swap} tells us that we have a $\ZZ_2$ action which interchanges the two
totally geodesic  $\CP_1$'s which generate $H^2(X,\ZZ)$. 
The intersection form must therefore be given 
 by 
 $$\left(\begin{array}{cc}-k & 1 \\1 & -k\end{array}\right)$$
 for some $k\geq 2$, and  $\Gamma = \ZZ_{k^2-1}$. 
 But we have also just seen that $|\Gamma|\leq 5$, 
 so it follows that 
  $k=2$ is 
 the only possibility. 
  \end{proof}

 \begin{lem}  \label{sprite} 
 Let $g_i$ and $(X, g_\infty)$ be as in Lemma \ref{nae}. 
 If $b_2(X)=1$, then $X$ must be diffeomorphic to the line bundle of 
 degree $-2$ or $-3$ over $\CP_1$. 
 \end{lem} 
\begin{proof}
By Lemma \ref{b21},    $X$ must be diffeomorphic to the line bundle of 
degree $-k$ over $\CP_1$ for some $k > 0$. 
If $C$ denotes the homology class of the zero section, the Poincar\'e
dual of $c_1$ is  the rational homology class $[(k-2)/k]C$, 
and it follows that the limit metric satisfies 
$$
\int_X |\mathring{r}|^2 d\mu_{g_\infty} = -8\pi^2 c_1^2 = 8\pi^2 \frac{(k-2)^2}{k} . 
$$
Since this number must be less that $8\pi^2$ by Lemma \ref{bub}, it follows that 
$k=2$ or $3$. 
\end{proof}

Thus, in light of Lemmas \ref{bub2}, \ref{bottom}, \ref{elf}, and \ref{sprite}, 
only three cases  still remain:
\begin{description}
\item{(i)} $b_2(X)=1$, $\Gamma =\ZZ_3$; 
\item{(ii)} $b_2 (X)=1$, $\Gamma = \ZZ_2$; and 
\item{(iii)} $b_2 (X) = 2$, $\Gamma = \ZZ_3$. 
\end{description}

The first of these, however, is easy to eliminate: 

\begin{lem} In the situation of Lemma \ref{nae},
$X$ cannot be as in  case {\em (i)} above. 
\end{lem} 
\begin{proof} 
Suppose $X$ were as in case (i). Then, by Lemmas \ref{swap} and \ref{sprite},
$M= \CP_2\# 2\overline{\CP}_2$ would contain  a smoothly embedded 
$2$-sphere $S$ of self-intersection $-3$ whose homology class was
invariant under $F_1\leftrightarrow F_2$. This 
$\ZZ_2$-invariance of 
$[S]\in H_2 (M, \ZZ)$  
would then allow us to express this homology class as 
$$[S] = m F_1 + m F_2 + n E$$
for some integers $m$ and $n$, and the self-intersection condition would then become
$$
-3= 2m^2 - n^2. 
$$
But reducing this mod $3$ gives us 
$$m^2 + n^2 \equiv 0 \bmod 3.$$
Since $m^2, n^2 \equiv 0 \mbox{ or } 1 \bmod 3$, this can only happen if 
$m, n \equiv 0\bmod 3$. But now setting $m=3 j$, $n=3k$, we then 
have 
$$-1= 6j^2 - 3 k^2$$
and reducing mod $3$ again then yields a contradiction. 
Thus case (i) cannot arise. 
\end{proof}

Eliminating the remaining two cases is not much harder, but 
does use considerably more of the information available to us. 

\begin{lem} In the situation of Lemma \ref{nae},
$X$ cannot be as in either of cases {\em (ii)} or {\em (iii)}  above. 
\end{lem} 
\begin{proof}
Since  the limit metric $g_\infty$ on $X$ is by construction a pointed limit 
 of larger and larger rescalings of the metrics $g_i$,  the generators of 
 $H_2(X,\ZZ)$ must arise from smooth $2$-spheres $S_i\subset M$ 
 whose areas with respect to the $g_i$ tend to zero as $i\to \infty$. 
   In case 
  (ii), let $S_i$ be the smooth $2$-sphere corresponding
  to the zero section $\CP_1$; in case (iii), let 
  $S_i$ be a $2$-sphere corresponding to one of the 
  two $\CP_1$ generators, and $\tilde{S}_i$ be its reflection under
  $F_1\leftrightarrow F_2$. Now 
  any of our unit-volume bilaterally symmetric K\"ahler classes $[\omega_i]$
  is  of the form 
  $$[\omega_i] = \frac{(1+x_i) (F_1+F_2) - x_i E}{\sqrt{1+2x_i + \frac{x_i^2}{2}}}$$
  and has ${\mathcal A}([\omega ]_i) = f(x_i) < 8$. Since $\lim_{x\to \infty} f(x) =9$, 
  we have 
  $x_i \in (0,K)$ for some fixed upper bound $K$.
 Choose some $i$ large enough so that the associated embedded 
  $2$-sphere $S_i\subset M$ has area $< (K+1)^{-1}$
 with respect to $g_i$.  In case (ii), we then set $\Sigma = [S_i]\in H^2 (M, \ZZ)$,
 while 
 in case (iii), we set $\Sigma = [S_i]+ [\tilde{S_i}]$. 
 Since
 $\Sigma$ is then  either represented $S_i$ or by $S_i$ together with  its reflection 
 $\tilde{S_i}$, we then have
 \begin{equation}
\label{tiny}
|[\omega_i]\cdot \Sigma | < 2 \area (S_i) < \frac{2}{K+1} 
\end{equation}
  by Wirtinger's inequality.

 Since the homology class $\Sigma$ is  $\ZZ_2$-invariant,  we have 
  $$\Sigma = j F_1 + jF_2 + k E$$
  for some integers $j$ and $k$. But we have arranged that $\Sigma^2 = -2$
  in either case (ii) or case (iii),
  so we obtain
  $$
  -2 = 2j^2 -k^2
  $$
  and $k\neq 0$ is therefore even, while $j$ is odd. Setting 
  $k=2\ell$ for some integer $\ell\neq 0$, we therefore have 
  $$-1= j^2 -2\ell^2$$
  and hence
  $$\left(\frac{j}{\ell} \right)^2= 2-\frac{1}{\ell^2}.$$
  In particular, this tells us that 
  $$\left|\frac{j}{\ell} \right|\geq 1,$$
  so $j$ and $j+\ell$ cannot have opposite signs, and $j\neq 0$.    
Hence 
\begin{eqnarray*}
\left| 
[\omega_i]\cdot \Sigma \right| 
&=& 
 \frac{\left| 2j (1+x_i) + kx_i\right|}{\sqrt{1+2x_i + \frac{x_i^2}{2}}}  = 
2 \frac{\left|  j +( j+ \ell)x_i \right| }{\sqrt{1+2x_i + \frac{x_i^2}{2}}}
\\ &>& \frac{2}{\sqrt{1+2x_i + \frac{x_i^2}{2}}} >   
\frac{2 }{1+x_i}
> \frac{2}{1+K}
\end{eqnarray*}
since $x_i \in (0, K)$. Hence (\ref{tiny}) implies  the glaring contradiction  
$$\frac{2}{K+1} > \left| [\omega_i]\cdot \Sigma \right| > \frac{2}{K+1}, $$
and it  follows that  cases (ii) and (iii)  never actually arise. 
  \end{proof}
  
  Since all possible deepest bubbles  have thus  been excluded,  no bubbling can occur, and 
  Theorem \ref{leap} therefore implies the following: 
  
  \begin{prop} \label{flex} 
  Let $g_i$ be a sequence of unit-volume bilaterally symmetric extremal K\"ahler metrics 
  on $(M,J)=\CP_2 \# 2\overline{\CP}_2$ such that the corresponding 
  K\"ahler classes $[\omega_i]$ all satisfy ${\mathcal A}([\omega_i]) \leq 8-\delta$
  for some $\delta > 0$.  
  Then there is a subsequence $g_{i_j}$  of metrics and a sequence of 
  diffeomorphisms $\Phi_j: M \to M$ such that
  $\Phi_j^*g_{i_j}$ converges in the smooth topology to an 
  extremal K\"ahler metric on the smooth $4$-manifold $M$ 
  compatible with some complex structure
  $\tilde{J}= \lim_{j\to \infty}\Phi_{j*}J$. 
      \end{prop}

\section{The Proof of Theorem \ref{big}} 

In the previous section, we saw that sequences of 
bilaterally symmetric extremal K\"ahler metrics with 
${\mathcal A} < 8-\delta$ necessarily have subsequences which
converge as smooth metrics. We now use this to 
study the set of K\"ahler classes which admit extremal K\"ahler metrics. 

For any positive real number $x>0$, let
 $[\omega]_x$ denote the K\"ahler class $(1+x)(F_1+F_2) -x E$
 on $(M,J)=\CP_2 \# 2\overline{\CP}_2$, 
and let us once again consider the function 
$f(x)={\mathcal A}([\omega]_x)$
studied in \S \ref{action}. Set  $L$ denote the smallest positive number in
$f^{-1}(8)$, so that 
$(0,L)$ is a connected component\footnote{ 
We are actually convinced that there are no
other connected components, but  there is no pressing need to 
try to prove this here!} 
 of $f^{-1}[(0,8)]\cap \RR^+$. 
 
\begin{thm}\label{freeze} 
For every $x\in (0,L)$,  $[\omega]_x=(1+x)(F_1+F_2) -x E$
is the K\"ahler class of  an extremal K\"ahler metric on the complex
surface $M=\CP_2 \# 2\overline{\CP}_2$ obtained by blowing up $\CP_2$ at two distinct points. 
\end{thm}
\begin{proof}
Consider the subset ${\mathcal X}$ of the interval $(0,L)$ consisting of those 
$x$ for which 
$[\omega]_x$ contains an extremal K\"ahler metric. Then  
${\mathcal X}$ is non-empty \cite{arpasing} and open \cite{ls2}. Since
$(0,L)$ is connected, it therefore suffices to show that ${\mathcal X}$ is also closed. 

To this end, consider a sequence $x_i\in {\mathcal X}$ which converges to some 
$x\in (0,L)$. Consider the corresponding extremal K\"ahler metrics
$g_i$, which have volume $[\omega ]_{x_i}^2/2 = 1 + 2x_i+x_i^2/2 \to
1 + 2x+x^2/2 = [\omega ]_x^2/2$. By rescaling these to unit volume, 
applying Proposition 
 \ref{flex}, and then rescaling back, there 
must exist a subsequence $g_{i_j}$ of the $g_i$  and   a 
sequence of diffeomorphisms $\Phi_j$ such that the pull-backs
$\Phi_j^*g_{i_j}$ smoothly converge to  
 a metric $g$ on $M$. Now recall that each of the metrics $g_i$ is 
toric, for a fixed action of the 2-torus on $M$. Choose a fixed point $p$ of this action, and 
choose an decomposition of the tangent space $T_p$ into a direct sum 
$L_1\oplus L_2$ into two complex lines 
which diagonalize the induced action of 
consider its images $\Phi_j(p)$ under these diffeomorphisms. Since $M$ is 
compact, we may assume that these points converge to a point $\hat{p}$ in $M$;
similarly, by again passing to a subsequence, we may also assume that the  
images of the orthogonal subspaces $L_1, L_2 \subset T_pM$ converge to 
give an orthogonal decomposition of $T_pM$. Once this is done, we then 
obtain a limit isometric action of the $2$-torus on $(M,g)$ by 
pushing forward the  corresponding  rotations of $T_pM$ and  conjugating with 
the exponential map of $g$. Since the push-forwards $\Phi_{j*}J$ converge to 
a complex structure $\tilde{J}$ which is parallel with respect to $g$, we moreover
conclude that this limit torus action is holomorphic with respect to $\tilde{J}$.

Now each of the holomorphic curves $F_1$, $F_2$ and $E$ in 
$\CP_2\# 2\overline{\CP}_2$ is the fixed point set of the isometric
action of some circle in the $2$-torus, so each is totally geodesic
with respect to the $g_i$. By looking at the corresponding fixed point
sets of the limiting action of circle subgroups, we can therefore find totally 
geodesic $2$-spheres in $(M,g)$  which are the limits of the images of these
submanifolds. These limit $2$-spheres are moreover holomorphic 
curves with respect to $\tilde{J}$, and have the same homological 
intersection numbers as the original curves $F_1$, $F_2$ and $E$.  
By blowing down the image of $E$ and applying surface classification, 
we thus conclude that  $(M, \tilde{J})$ is biholomorphic
to the blow-up of $\CP_1\times \CP_1$ at a point. Moreover, since the 
areas of these totally geodesic $2$-spheres are the limits of the
areas of the corresponding $\CP_1$'s with respect to the $g_{i_j}$,
the K\"ahler class of $g$ on $(M,\tilde{J})$ must be the limit of
the $[\omega ]_{i_j}$. Thus there exists a diffeomorphism 
$\Phi: M\to M$ such that $\Phi^*\tilde{J}= J$, and such that 
$\Phi^*g$ becomes an extremal K\"ahler metric with K\"ahler class
$[\omega ]_x$. This shows that $x\in {\mathcal X}$. Thus ${\mathcal X}$ is closed,
and the result follows. 
\end{proof}

Now Lemma \ref{magna} tells us that $x_0\in (0,L)$. It therefore follows that 
$$[\omega ]_{x_0}= (1+x_0)(F_1+F_2)- x_0 E$$
 is the K\"ahler class of an extremal K\"ahler metric $g$. 
However,  Corollary \ref{suite} then tells us that the conformally related
metric $h=s^{-2}g$ is Einstein, and defined on all of $M$. We have
therefore proved the existence of an Einstein metric on 
$\CP_2\# 2\overline{\CP}_2$ which is conformally K\"ahler, and therefore Hermitian, 
precisely as 
claimed by  
 Theorem \ref{big}. 
 
 \section{Concluding Remarks} 
 
 While we have proved the existence of the Einstein metrics promised 
 by Theorem \ref{big} and its corollaries, we have not proved that
such metrics are  necessarily {\em unique} up to rescaling. On the other hand, 
 in light  of \cite{xxgang}, this would follow \cite{spccs} if 
 the critical points of ${\mathcal A}$ could simply be shown to form a unique ray in the
 K\"ahler cone. Extensive  computer calculations by
 Gideon Maschler \cite{gideon} provide  overwhelming evidence 
 to this effect, and  could arguably be called a  ``computer-assisted proof''
 of this assertion. Nonetheless, it might be wiser to simply treat Maschler's 
 calculations as strong circumstantial evidence, rather
 than  as a definitive proof.  
 In any case, a conceptual proof more accessible  to the human mind would be 
prerequisite to any claim that we really understand
this phenomenon. 
 One valiant  attempt in this regard 
 was made by Simanca and Stelling \cite{simstelling}, who calculated the 
 Hessian of the functional and concluded that all critical points must be local minima;
but, contrary to what is tacitly assumed  in their  paper, this  alone does not
logically  suffice to show 
  that the critical ray  is actually  unique. We would therefore like to draw 
 attention to this important gap in our knowledge,  in the hope  that some 
 interested reader will be inspired to provide a definitive solution to the uniqueness
 problem.

 Of course, the results of  this paper also prove the existence of 
 extremal K\"ahler metrics in a whole range of bilaterally symmetric
 K\"ahler classes on $\CP_2 \# 2\overline{\CP}_2$ other than the `target' class $x=x_0$ 
 used to 
 construct our conformally K\"ahler Einstein metric. For example,
 it is not difficult to show that $x=1$  actually lies in the interval $(0,L)$
 of Proposition 
 \ref{freeze}, and this
 implies that  
 the first Chern class $c_1(M)$  is actually the K\"ahler class of an extremal K\"ahler metric. 
 In fact, in light of Theorem \ref{leap}, it seems plausible to us that {\em every} 
 bilaterally symmetric K\"ahler class might  be represented by 
 such a metric, but  one would certainly need to consider many more
 possible  bubbling modes  as  $x\to \infty$. Nonetheless, since the results
 of Arezzo-Pacard-Singer \cite{arpasing} do imply the existence of such metrics
 for all sufficiently  large $x$,  such a  conjecture  might seem quite tempting. 
Of course, it would also be highly desirable to understand existence  for
K\"ahler classes which do not satisfy our  convenient but  
somewhat  arbitrary  condition of  bilateral symmetry. However,
it is not hard to check  that large regions of the K\"ahler cone of   $\CP_2 \# 2\overline{\CP}_2$ 
actually lie outside   the controlled cone of \S \ref{sobstory}, so  our method of controlling 
Sobolev constants, leading to Theorem \ref{leap}, actually exploited the imposition of 
 bilateral symmetry in an essential manner. 

The key r\^ole of toric geometry in the present paper may make it 
seem curious that we have not consistently operated in the toric context 
throughout, rather than taking limits which are  only then proved to be 
toric at the price of considerable extra effort.  Since Donaldson \cite{dontor} has outlined
a  beautiful, systematic  program for the study of toric extremal K\"ahler manifolds, 
we certainly wonder if some  steps in our long argument could be simplified
or eliminated altogether through the adoption of a different point of view!

We would like to once again draw the reader's attention
 to the central r\^ole played by ALE scalar-flat K\"ahler
 surfaces in our proof. Although there is a considerable 
 literature \cite{calein,calhk,caldsing,gibhawk,hitpoly,
 rejoyce,kron,lpa,mcp2,lebmero}
concerning the construction of such metrics, it is apparent that  too little is 
still known about their classification outside the hyper-K\"ahler realm  so  thoroughly
mapped out  by 
 Kronheimer \cite{krontor}. 
In general, this problem seems  daunting, but in the toric
case it might  be feasible to prove that the only possibilities 
are the metrics   constructed explicitly by Calderbank and Singer \cite{caldsing}. 
A related problem would be to try to classify toric 
anti-self-dual orbifolds by extending the beautiful paper
of Fujiki \cite{fujiki}. 

Finally, we believe that it would be interesting to 
extend the techniques used in this paper to 
construct  Bach-flat K\"ahler  metrics which are not
globally conformally Einstein. Such metrics
can certainly sometimes exist  when $c_1$ fails to be positive;
for example, the 
study of   extremal K\"ahler metrics on Hirzebruch surfaces  \cite{hwasim} 
reveals   that  the differentiable manifold
$S^2\times S^2$ admits Bach-flat conformal 
structures  corresponding to many different critical values
of the Weyl action ${\mathcal W}$.  It would be certainly be interesting  
to see if  this same phenomenon occurs for many  other  complex surfaces.
We hope   that it may prove possible to use our present methods
 to construct such metrics on certain other  surfaces 
with  $c_1^2 > 0$. In any case, the  convergence of sequences of 
extremal K\"ahler metrics on    complex surfaces seems destined to 
develop into a rich subject of 
broad  interest,  ultimately involving    issues  of 
 algebro-geometric  stability \cite{donaldsonk1,dontor,mabstab,ross,rotho,tianstab} 
that have played no r\^ole at all in our present story. 

\vfill  
 
\noindent {\bf Acknowledgments:}
The first author  gratefully  thanks   S.K.  Donaldson 
for many enlightening discussions  of extremal K\"ahler metrics, and the 
Princeton Mathematics Department for its hospitality
during the writing of this article.  
 The second author would like to thank Michael Taylor
 for some helpful comments regarding regularity issues, and 
 Michael Anderson for some 
 useful pointers   concerning Gromov-Hausdorff convergence.

 \pagebreak

\vfill

{\footnotesize 
\noindent
{\sc Xiuxiong Chen},
{University of Wisconsin-Madison\\
Mathematics Department, 
480 Lincoln Dr, 
Madison WI 53706-1388}\\
{\sc e-mail}: xiu@math.wisc.edu
 \\ 
 \\
{\sc  Claude LeBrun}, State University of New York at Stony Brook\\
{Department of Mathematics, SUNY, 
Stony Brook, NY 11794-3651}\\
{\sc e-mail}: claude@math.sunysb.edu
\\
\\
{\sc Brian Weber},
{University of Wisconsin-Madison\\
Mathematics Department, 
480 Lincoln Dr, 
Madison WI 53706-1388}\\
{\sc e-mail}: weber@math.wisc.edu}

\end{document}